\documentclass[12pt,reqno]{amsart}
\usepackage{amsthm}
\usepackage{fullpage,url}
\usepackage[off]{auto-pst-pdf}
\usepackage{hyperref} 
\usepackage{amscd}   
\usepackage[all,graph,poly, knot]{xy} 
\usepackage{amssymb,amsmath}
\usepackage{pstricks,pst-plot,pst-node}
\usepackage{multicol}
\usepackage{setspace}
\DeclareFontEncoding{OT2}{}{} 

\providecommand{\abs}[1]{\lvert#1\rvert}

\newcommand{\cA}{\mathcal{A}}
\newcommand{\CC}{\mathbb{C}}
\newcommand{\cC}{\mathcal{C}}

\newcommand{\be}{\mathbf{e}}

\newcommand{\op}[1]{\operatorname{#1}}

\newcommand{\QQ}{\mathbb{Q}}

\newcommand{\ZZ}{\mathbb{Z}}

\swapnumbers
\newtheorem{theorem}{Theorem}[section]
\newtheorem{lemma}[theorem]{Lemma}
\newtheorem{corollary}[theorem]{Corollary}

\theoremstyle{definition}

\theoremstyle{remark}

\newtheoremstyle{head}
{}
{}
{\bfseries}
{}
{}
{}
{.5em}
{}

\theoremstyle{head}

\newtheoremstyle{citing}
  {3pt}
  {3pt}
  {\itshape}
  {}
  {\bfseries}
  {:}
  {.5em}
  {\thmnote{#3}}

\theoremstyle{citing}
\begin{document}
\title[]{Indicators of Tambara-Yamagami categories and Gauss sums}
\author{Tathagata Basak}
\address{Department of Mathematics\\Iowa State University \\Ames, IA 50011}
\email{tathagat@iastate.edu}
\urladdr{}
\author{Ryan Johnson}
\address{Department of Mathematics\\Grace College \\Winona Lake, IN, 46590}
\email{johnsor@grace.edu}
\urladdr{}
%
%
\date{July 24, 2014}
\begin{abstract} 
We prove that the higher Frobenius-Schur indicators, introduced by
Ng and Schauenburg, give a strong enough invariant to distinguish between any two
Tambara-Yamagami fusion categories.
Our proofs are based on computation of the higher indicators in terms of
Gauss sums for certain quadratic forms on finite abelian groups and rely on the 
classification of quadratic forms on finite abelian groups, due to Wall.
\par
As a corollary to our work, we show that the state-sum invariants of  a Tambara-Yamagami category
determine the category as long as we restrict to Tambara-Yamagami categories coming from groups $G$
whose order is not a power of $2$. Turaev and Vainerman proved this result under the
assumption that $G$ has odd order and they conjectured that a similar result should hold
for groups of even order. 
We also give an example to show that the assumption that $\abs{G}$ is not a power of 2, cannot be
completely relaxed. 
\end{abstract}
\maketitle
%
%
%
%
\section{Introduction}
\label{section-introduction}
%
%
Fusion categories (see \cite{ENO:OFC}) occur in various branches of mathematics:
low dimensional topology, subfactors, and quantum groups, to name a few. Classification of 
fusion categories, although currently out of reach in general, is a main driving question in the area.
A natural method for classifying objects in mathematics is via numerical invariants.
In \cite{NS:FS}, Ng and Schauenburg introduced a class of invariants of spherical pivotal fusion categories
(to be simply called spherical categories) called the higher Frobenius-Schur indicators.
Let $\cC$ denote a spherical category.
For each simple object $V$ of $\cC$ and each integer $k \geq 1$, Ng and Schauenburg define a
complex number $\nu_k(V)$, called the $k$-th indicator of $V$. These build on and generalize
many previous works, e.g. \cite{Ba:FS}, \cite{FGSV:S}, \cite{FS:C}, \cite{KSZ:H}, \cite{LM:FS}, \cite{MN:C};   
we refer the reader to the introduction of \cite{NS:FS} for more details.
For $k = 2$, these invariants generalize the classical Frobenius-Schur indicator of a finite group
representation.
The Frobenius-Schur indicators of the simple objects of $\cC$ can be used to define
the Frobenius-Schur exponent of $\cC$, denoted $\op{FSexp}(\cC)$. When $\cC$ is the
representation category of a quasi-Hopf algebra, $\op{FSexp}(\cC)$ is equal to 
$\exp(\cC)$ or $2 \exp(\cC)$ (\cite{NS:FSE}, theorem 6.2) where
$\exp(\cC)$ denotes the exponent of $\cC$ in the sense of Etingof et.al. (see \cite{E:V} and its references).
\par
The higher indicators are powerful tools for studying pivotal categories.
For example, they were used in \cite{NS:CS} to prove that the projective
representation of $\op{SL}_2(\ZZ)$ obtained from a modular tensor category
factors through a finite quotient $\op{SL}_2(\ZZ/n \ZZ)$ for some $n$. 
In this article we demonstrate that the numbers $ \nu_k( V)$, as $k$ varies
over natural numbers and $V$ varies over the set of simple objects of $\cC$,
give a strong enough numerical invariant of $\cC$ that is able to distinguish
between  any two spherical categories in an interesting class, 
known as Tambara-Yamagami categories ($\op{TY}$-categories for short).
\par
Before introducing the $\op{TY}$-categories and stating our theorem
precisely, we want to put our results in context.
Susan Montgomery  has asked whether the FS-indicators of a semsimple
Hopf algebra determine the tensor category of its representations. 
This was shown to be true for the class of semisimple Hopf algebras of dimension
$8$ in \cite{NS:CI}. The representation categories of these Hopf algebras are $\op{TY}$-categories.
In \cite{KMN:T} it was shown that for the class of non-semisimple Hopf algebras called Taft algebras,
the second indicator can distinguish between the finite tensor categories of their representations.
Along similar lines, Siu-Hung Ng has asked whether a spherical fusion category generated by a
simple object is completely determined by it FS-indicators
(Siu-Hung Ng, private communications). 
Our results give an affirmative answer to this question for the class of $\op{TY}$-categories. 
\par
Let $G$ be a finite group. Let $S$ be a finite set which contains $G$ and one extra element, denoted $m$.
Consider the following fusion rule on $S$: 
\begin{equation*}
g \otimes h = gh,  \; \; m \otimes g = g \otimes m = m, \; \; m \otimes m = \bigoplus_{x \in G} x \; 
\text{ \; for all \;}   g, h \in G.
\end{equation*}
In \cite{TY:TC}, Tambara and Yamagami classified all fusion categories that have the above fusion rule;
for a conceptual proof of this classification see \cite{ENO:FCH}, example 9.4.
Such fusion categories exist only if $G$ is abelian and are classified by pairs $(\chi, \tau)$ where
$\chi: G \times G \to \CC^*$ is a non-degenerate symmetric bicharacter on $G$ and
$\tau$ is a square root of $\abs{G}^{-1}$.
For each tuple $(G, \chi, \tau)$ as above, there exists a spherical  category, denoted
$\op{TY}(G, \chi, \tau)$. Two $\op{TY}$-categories $\cC = \op{TY}(G, \chi, \tau)$
and $\cC' = \op{TY}(G', \chi', \tau')$ are isomorphic as spherical categories if and only if $\tau = \tau'$ and
$(G, \chi ) \simeq (G', \chi')$, that is, there exists an isomorphism $f: G \to G'$ such that
$\chi'( f(x), f(y)) = \chi(x, y)$ for all $x, y \in G$. Let $\op{Irr}(\cC) = G \cup \lbrace m_{\cC} \rbrace$ be the simple objects of $\cC$.
There is a canonical (spherical) pivotal structure on $\cC$ such that the pivotal dimension 
of an object matches the Frobenius-Perron dimension. 
For an object $V$ of $\cC$, let $\op{pdim}(V)$ denote its pivotal dimension for this canonical 
pivotal structure.
We shall prove the following theorem:
\begin{theorem}
Let $\cC$ and $\cC'$ be two $\op{TY}$-categories.  If 
\begin{equation*}
\sum_{V \in \op{Irr}(\cC)} \nu_k(V) = \sum_{V \in \op{Irr}(\cC')} \nu_k(V) \text{\; and \;}
\sum_{V \in \op{Irr}(\cC)} \op{pdim}(V) \nu_k(V) = \sum_{V \in \op{Irr}(\cC')}  \op{pdim}(V) \nu_k(V)
\end{equation*}
for all $k \geq 1$, then
$\cC \simeq \cC'$ as spherical fusion categories.
\label{t-indicators-determine-b}
\end{theorem}
Now we shall describe our plan for the proof of this theorem and give a summary of contents
of the sections. 
Let $\cC = \op{TY}(G, \chi, \tau)$ and $\cC' = \op{TY}(G', \chi', \tau')$ be two $\op{TY}$-categories.
Assuming $G$ and $G'$ are non-trivial groups, the assumptions theorem \ref{t-indicators-determine-b}
are quickly seen to be equivalent to
$\nu_k(m_{\cC} ) = \nu_k( m_{\cC'})$ and $\sum_{x  \in G} \nu_k(x) = \sum_{x \in G'} \nu_k(x)$.
Based on work done in \cite{S:FSI}, we can easily conclude that $G \simeq G'$ and  $\tau = \tau'$. 
Most of our work goes into showing that if $\nu_k(m_{\cC}) = \nu_k(m_{\cC'})$ for all $k$, then
$(G, \chi) \simeq (G, \chi')$.
Shimizu calculated $\nu_k(m_{\cC})$ (see \cite{S:FSI}, theorem 3.3, 3.4)
using an expression for the indicator in terms of the twist of the Drinfeld center of $\cC$
(\cite{NS:FSE} theorem 4.1).
This project started for us when Siu-Hung Ng asked us whether the eighth root of unity in
\cite{S:FSI} theorem 3.5 is related to the signature modulo $8$ for some
related lattice. This indeed turns out to be the case. A simple re-statement of
Shimizu's result gives us a formula relating the indicators
$\nu_{2k}(m_{\cC})$ to certain quadratic Gauss sums; see lemma \ref{l-nu-and-Theta}. 
This formula is the starting point for our calculations, and
we want to explain it in precise terms. For this we need some notation.
\par
Let $G$ be an abelian group, always written additively in this paper unless otherwise stated.
Let $q: G \to \QQ/\ZZ$ be a quadratic form on $G$. 
Given a pair $(G,q)$, one defines the associated quadratic Gauss sum 
\begin{equation}
\Theta( G, q) = \abs{G}^{-1/2} \sum_{x \in G} \be(q(x)), \text{\;\; where \;} \be(x) = e^{2 \pi i x}.
\label{e-def-Theta}
\end{equation}
For $k \in \ZZ$, it will be also convenient to define the invariant 
\begin{equation}
\xi_k(G, q) = \Theta(G, q)^k \Theta( G, -k \cdot q).
\label{e-def-xi}
\end{equation}
Let $\cC = \op{TY}(G, \chi, \tau)$ be a TY-category where $(G, \chi, \tau)$ is a triple as above.
We choose a quadratic form $q$ on $G$ such that
$\chi(x,y) = \be( -\partial q(x,y) )$ where $\partial q : G \times G \to \QQ/\ZZ$ denotes the symmetric $\ZZ$-bilinear form
\begin{equation}
\partial q (x, y) = q(x + y)  - q(x) -  q(y).
\label{eq-def-partial-q}
\end{equation}
One can show that such a $q$ always exists. In lemma \ref{l-nu-and-Theta},
we prove that for $k \geq 1$,
\begin{equation*}
\nu_{2k}( m_{\cC} ) = \op{sign}(\tau)^k \xi_k(G, q).
\end{equation*}
Much of the calculation in sections \ref{section-Gauss-sum} and \ref{section-TY} is geared towards
finding explicit formulae for $\xi_k(G,q)$ by using the classification
of the irreducible quadratic forms and the known values of Gauss sums of these irreducible forms.
The calculations are more complicated when $G$ is a $2$-group, 
which is a well known feature in the theory of quadratic forms on finite abelian groups.
When $G$ is a $2$-group,  and $v_2(k)$ (the two-valuation of $k$) is at least 1, we relate 
$\xi_k(G,q)$ to an invariant $\sigma_{v_2(k)}(\partial q)$ of the the pair $(G,\partial q)$
(see lemma \ref{l-two-power-ind}).
The invariant $\sigma_n(\partial q)$ is a generalization of Karviere-Brown-Peterson-Browder invariant, see
\cite{Br:G} and page 33 of \cite{KK:LP}. Detailed calculation of the values of the Gauss sums and properties of 
the invariant $\sigma_n(\partial q)$ lets us conclude that the bicharacter $\chi$ can be recovered
from values of the  Gauss sums, thus proving our theorem.
\par
Sections \ref{section-bq} through \ref{section-indicator} contain preparatory material.
In section \ref{section-bq}, we collect the background material necessary for quadratic and bilinear
forms on finite abelian groups and their classification.
The results here are mostly due to C.T.C.Wall \cite{W:QF}; also see
\cite{MI:SB}, \cite{KK:LP}, \cite{VVN:ISBF} and the proofs can be found in these references. 
However, we have chosen to include the proofs of most of what we need in a detailed appendix. In particular we give a
proof of the existence part of Wall's theorem (See theorem \ref{t-wall}) on the classification of non-degenerate quadratic and bilinear forms 
on finite abelian groups. 
We have explained our reason for including the appendix in section \ref{section-bq}, following the statement of
theorem \ref{t-wall}.
\par
Section \ref{section-Gauss-sum} contains the background on values of Gauss sums and calculation
of $\xi_k(G,q)$ in various cases.
Section \ref{section-indicator} introduces the $\op{TY}$-categories and relates the indicator values
$\nu_{2k}(\cC)$ with Gauss sums. With these preparations, we prove 
theorem \ref{t-indicators-determine-b} in section \ref{section-TY}. 
\par
Finally, in section \ref{section-statesum} we apply theorem \ref{t-indicators-determine-b}, to address a
recent conjecture of Turaev and Vainerman \cite{TV:TY} regarding $3$-manifold invariants constructed from
$TY$-categories. 
Given a compact $3$-manifold $M$ and a spherical category $\cC$, one can
define an invariant $\abs{M}_{\cC}$, called the state-sum invariant, see \cite{TV:TY}.
In \cite{TV:OT} it was shown that $\abs{M}_{\cC} = \tau_{Z(\cC)}(M)$, where $Z(\cC)$ is the
Drinfeld center of $\cC$ and $\tau_{Z(\cC)}(M)$ denotes the Reshetikhin-Turaev invariant.
For $k \geq 1$, let 
$L_{k,1} = \lbrace (z_1,z_2) \in \CC^2 \colon \abs{z_1}^2 +\abs{z_2}^2 = 1\rbrace/\langle (z_1, z_2) \sim e^{2 \pi i/k} (z_1, z_2) \rangle $ 
denote the lens spaces. 
In theorem \ref{t-Lens}, we show that a $\op{TY}$-category 
$\cC = \op{TY}(G, \chi, \tau)$ is determined by the sequence of state-sum invariants 
$\lbrace \abs{L_{k,1}}_{\cC} \colon k \geq 1 \rbrace$ as long as we restrict to categories such that
$\abs{G}$ has an odd factor. Turaev and Vainerman proved this result 
assuming that $\abs{G}$ is odd and conjectured that a similar result should hold for groups
of even order. In section \ref{section-statesum}, we exhibit two non-isomorphic tuples 
$(G, \chi, \tau)$ and $(G', \chi', \tau')$ such that $ \abs{L_{k,1}}_{\op{TY}(G, \chi, \tau) } =  \abs{L_{k,1}}_{\op{TY}(G', \chi', \tau')}$
for all $k$. In our example, both $G$ and $G'$ have order $64$.
This example demonstrates that one needs to put some hypothesis on the possible orders of $G$, or else consider state-sum invariants of other 3-manifolds if one has to recover the category from the data of these invariants.
\par
Quadratic and bilinear forms on finite abelian groups appear in various places in topology
and geometry. We give some examples:
\begin{itemize}
\item The ``torsion linking pairing'' on the torsion part of the $n$-th integral homology of a $(2n + 1)$ dimensional
real compact manifold coming from Poincar\'e duality and intersection pairing, for example, see \cite{KK:LP}.
For $3$-manifolds we get a pairing on the torsion $1$-cycles related to the linking number.
For this reason, discriminant forms are called linking pairs in \cite{KK:LP}.
\item Intersection pairing on the torsion part of middle cohomology of a $(4n + 2)$ dimensional manifold and
computation of Kervaire-Arf invariants, see \cite{Br:G}.
\item Study of integral lattices coming from algebraic geometry, for example study of $K_3$ surfaces,
see \cite{VVN:ISBF}. Let $G$ be a finite abelian group and $b$ be a non-degenerate symmetric bilinear form on $G$.
For each pair $(G,b)$, there exists a pair $(L,B)$, where 
$L \simeq \ZZ^n$ and $B: L \times L \to \ZZ$ is a non-degenerate symmetric $\ZZ$-bilinear form
such that $G = L'/L$ and $b$ is the $\QQ/\ZZ$ valued form induced on $L'/L$ by
$B$; here $L'$ denotes the dual lattice of $L$. For this reason we have borrowed the name 
``discriminant form'' from \cite{VVN:ISBF} for pairs $(G,b)$.
\end{itemize}
We hope that the methods of calculation of Gauss sums will have other uses
in computations of Gauss sums coming from the above sources. 
\par
{\bf Acknowledgment:} This work would not be possible without the guidance of Siu-Hung Ng during the
inception of the project. We are grateful to him for suggesting the problem for this project.
We would also like to thank him for his encouragement, many useful conversations, and for pointing
out many references. We would like to thank both our referees for thoughtful reviews.
The referee's suggested revision of an earlier draft had lead to a lot of simplification of
our previous proof and a signification reduction in the preparatory lemmas needed.
%
%
%
%
\section{Bilinear and quadratic forms on finite abelian groups}
\label{section-bq}
%
%
{ \bf Definition.} Let $G$ be a finite abelian group (written additively). Let $\exp(G)$ denote the exponent of $G$.
A {\it discriminant form} is a pair $(G,b)$ where $G$ is a finite abelian group and $b: G \times G \to \QQ/\ZZ$ is a symmetric bilinear form on $G$.
As all the bilinear forms considered in this article are symmetric,  the adjective ``symmetric" will sometimes be dropped. 
 Say that $b$ or $(G,b)$ is {\it non-degenerate} if for each nonzero $x \in G$ there exists $y \in G$ such that $b(x,y) \neq 0$.
\par
Let $G$ be a finite abelian group and $q$ be a quadratic form on $G$. 
We say that the pair $(G,q)$ is a {\it pre-metric} group.
We say that $q$ is non-degenerate and $(G,q)$ is a {\it metric group} if the bilinear form $\partial q$ (see eq. \eqref{eq-def-partial-q}) 
is non-degenerate. 
\par
The morphisms in the categories of discriminant forms and pre-metric groups are defined as usual.
Isomorphisms are often called isometries. There is an obvious notion of orthogonal direct sum on 
discriminant forms and pre-metric groups. If $(G_1, q_1)$ and $(G_2, q_2)$ are two pre-metric groups,
we let $(G_1, q_1) \bot (G_2, q_2)$ denote their orthogonal direct sum. 
The map $(G,q) \mapsto (G, \partial q)$ defines a functor from the category of  pre-metric groups (resp. metric groups) to
the category of discriminant forms (resp. non-degenerate
discriminant forms).
\newline
\par
{\bf Remark.}
Let $G$ be a finite abelian group. 
Note that a bilinear form on $G$ takes values in $\exp(G)^{-1} \ZZ/ \ZZ$.
Let $(G,q)$ be a pre-metric group. Let $a \in G$. Note that $\partial q(a,a) = 2 q(a)$, and so $q$ takes value in $(2 \exp(G))^{-1} \ZZ/ \ZZ$. If $G$ has odd order, then
 $a = 2(\tfrac{ \exp(G)+1}{2})a$. So $q( a) = (\tfrac{ \exp(G) + 1}{2})  \partial q(a,a)$.
Hence $q$ actually takes value in $ \exp(G)^{-1} \ZZ/ \ZZ$ and $\partial q$ determines $q$.
But this fails for groups of even order.
For example, consider the non-degenerate bilinear form on $\ZZ/ 4 \ZZ$ given by $b(x,y) = x y/4$. 
Then $q(x) = x^2/8$ and $q'(x) = 5 x^2/8$ are two distinct quadratic forms on $\ZZ/4\ZZ$
such that  $\partial q =\partial q' = b$. 
\newline
\par
{ \bf Definition.} Let $p$ be a prime. If $a$ is a rational number, $v_p(a)$ will denote the $p$-valuation of $a$.
It will be convenient to extend the definition of $p$-valuation as follows. Let $G$ be an abelian $p$-group.
Define $v_p: G \to \ZZ_{\leq 0} \cup \lbrace \infty \rbrace$ by $v_p(x) = - \log_p( \op{order}(x))$ 
if $x $ is a non-zero element of $G$, and $v_p(0) = \infty$.
We say that $v_p(x)$ is the {\it $p$-valuation} of $x$.
\par
This definition of $p$-valuation is useful to us because of the following example.
Let $\QQ_{(p)}$ be the ring of all rational numbers of the form $m/p^r$ where $m \in \ZZ$ and $r \in \ZZ_{\geq 0}$. 
If $(G, q)$ is a pre-metric $p$-group, then observe that $q$ and $\partial q$ takes values in the $\ZZ$-module 
$\QQ_{(p)} / \ZZ$.  
If $\alpha$ is a non-zero element of $\QQ_{(p)} / \ZZ$, then it can be written as $p^{-n} a$ for some $a \in \ZZ$
relatively prime to $p$. One has $v_p( \alpha) = -n$.
\newline
\par
Let $(G, b)$ be a discriminant form. Let $e_1, \dotsb, e_k \in G$ and $b_{i j} = b(e_i, e_j)$. 
The matrix $B = ( \!( b_{ i j} ) \!)$ is called the {\it Gram matrix} of $e_1, \dotsb, e_k$. 
We shall write $\op{Gram}_b( e_1, \dotsb, e_n) = B$. One has
\begin{equation*}
b\left( \sum_i g_i e_i, \sum_j h_j e_j\right) = (g_1, \dotsb, g_k) B (h_1 ,  \dotsb, h_k)^{tr}
\text{\; for all  \;} g_1,...g_k, h_1,...,h_k \in \ZZ.
\end{equation*}
A discriminant form (resp. pre-metric group) is called {\it irreducible} if it cannot be
written as an orthogonal direct sum of two non-zero discriminant forms (resp. pre-metric groups).
A finite abelian group is {\it homogeneous} 
if it is isomorphic to $(\ZZ/ p^r \ZZ)^n$ for some prime $p$ and positive integers $r$ and $n$.
For a $p$-group $G$, we let $\op{rk}(G)$ denote the minimum number of generators for $G$
or equivalently $\op{dim}_{\mathbb{F}_p} (G/ \Phi(G))$ where $\Phi(G)$ is the Frattini subgroup of $G$.
In particular
\begin{equation*}
\op{rk}((\ZZ/ p^r \ZZ)^n) = n.
\end{equation*}
An element of $(\ZZ/ p^r \ZZ)^n$ will often be written as a vector whose entries
come from $\ZZ/ p^r \ZZ$. A discriminant form on a homogeneous finite abelian group
will be often written down as $( (\ZZ/ p^r \ZZ)^n ,  B )$ where $B$ is a $n \times n$ matrix
with entries in $p^{-r} \ZZ/ \ZZ$ such that $b( x, y) = x B y^{tr}$ for all $x, y \in (\ZZ/ p^r \ZZ)^n$.
Let $p$ be an odd prime and $u_p$ denote a quadratic non-residue modulo
$p$. Table \ref{table-1} lists the irreducible metric groups $(G, q)$ and corresponding
irreducible discriminant forms $(G, \partial q)$.
\begin{table}
\centerline{
\begin{tabular}{|c|c|c|}
\hline
name (from \cite{MI:SB})  & $(G, q)$ &  $(G,\partial q)$ \\
\hline
$A_{p^r}$  & $\Bigl( \ZZ/ p^r \ZZ, q(x) = \frac{(p^r + 1)/2}{p^r} x^2 \Bigr) $ &  $\Bigl(\ZZ/p^r \ZZ, \frac{1}{p^r} \Bigr)$  \\ 
\hline
$B_{p^r}$ & $\Bigl( \ZZ/ p^r \ZZ, q(x) = \frac{u_p(p^r + 1)/2}{p^r} x^2 \Bigr) $ &  $\Bigl(\ZZ/p^r \ZZ, \frac{u_p}{p^r} \Bigr)$  \\ 
\hline
\hline
$A_{2^r}$  & $\Bigl( \ZZ/ 2^r \ZZ, q(x) = \frac{1}{2^{r+1}} x^2 \Bigr) $ &  $\Bigl(\ZZ/2^r \ZZ, \frac{1}{2^r} \Bigr)$ \\ 
\hline
$B_{2^r}$  & $\Bigl( \ZZ/ 2^r \ZZ, q(x) = \frac{-1}{2^{r+1}} x^2 \Bigr) $ &  $\Bigl(\ZZ/2^r \ZZ, \frac{-1}{2^r} \Bigr)$  \\ 
\hline
$C_{2^r}$  & $\Bigl( \ZZ/ 2^r \ZZ, q(x) = \frac{5}{2^{r+1}} x^2 \Bigr) $  &  $\Bigl(\ZZ/2^r \ZZ, \frac{5}{2^r} \Bigr)$\\ 
\hline
$D_{2^r}$ & $\Bigl( \ZZ/ 2^r \ZZ, q(x) = \frac{-5}{2^{r+1}} x^2 \Bigr) $  &  $\Bigl(\ZZ/2^r \ZZ, \frac{-5}{2^r} \Bigr)$ \\ 
\hline
$E_{2^r}$  & $\Bigl((\ZZ/2^r \ZZ)^2 , q(x_1, x_2) =  \frac{x_1 x_2}{2^r} \Bigr)  $  
& $\Bigl((\ZZ/2^r \ZZ)^2, \bigl( \begin{smallmatrix} 0 & 2^{-r} \\ 2^{-r} & 0 \end{smallmatrix} \bigr) \Bigr)$
\\ 
\hline
$F_{2^r}$ & $\Bigl((\ZZ/2^r \ZZ)^2 , q(x_1, x_2)  =  \frac{x_1^2 + x_1 x_2 + x_2^2}{2^r} \Bigr)  $  
&  $\Bigl((\ZZ/2^r \ZZ)^2, \bigl( \begin{smallmatrix} 2^{1 - r} & 2^{-r} \\ 2^{-r} & 2^{1 - r} \end{smallmatrix} \bigr) \Bigr)$ \\ 
\hline
\end{tabular}
} 
\caption{
Irreducible quadratic and symmetric  bilinear forms. In the first two rows of the table $p$ represents an odd prime. 
For the prime $2$ and for $r=1$ or $2$, some of the forms above are isometric.
For example $A_2 \simeq C_2$.}
\label{table-1}
\end{table}
\begin{theorem}[\cite{W:QF}, also see \cite{MI:SB}, \cite{VVN:ISBF}]
(a) Each non-degenerate discriminant form is an orthogonal direct sum of the irreducible discriminant forms listed in table \ref{table-1}.
\par
(b) Each metric group is an orthogonal direct sum of the irreducible metric groups listed in table \ref{table-1}.
\par
It follows that given any non-degenerate symmetric bilinear form $b$ on a finite abelian group $G$, there exists a quadratic form
$q$ on $G$ such that $\partial q = b$.
 \label{t-wall}
\end{theorem}
A proof of  theorem  \ref{t-wall} has been sketched in the appendix \ref{app-A}. Here we shall only give a brief indication of our argument.
This argument seems to be different from the proofs in the references above and we believe it is simpler. 
It  is probably well known to experts but we have not seen it spelled out in literature.
\par
Let $(G,b)$ be a discriminant form. Write $ G = \oplus_{p} G_{(p)}$ where $G_{(p)}$ is the $p$-Sylow subgroup of $G$.
Let $b_{(p)}$ be the restriction of $b$ to $G_{(p)} \times G_{(p)}$. It is easy to see that 
$(G, b)$ is an orthogonal direct sum of $(G_{(p)}, b_{(p)})$
as $p$ varies over primes. So it suffices to decompose $(G, b)$ into irreducibles when $G$ is a $p$-group for some prime $p$. 
\par
Let $G$ be a finite abelian $p$-group and $b$ be a non-degenerate symmetric bilinear form on $G$. 
The algorithm for decomposing $(G,b)$ into irreducibles boils down to diagonalizing symmetric matrices with
entries in $\QQ_{(p)}/\ZZ$ via conjugation.
The algorithm for diagonalization is the same as the well known algorithm for diagonalizing quadratic forms over
$p$-adic integers, see, for example, \cite{CS:SP} chapter 15, section 4.4. This algorithm is the core of our argument.
We repeat that we could not find this argument written out in literature for bilinear forms on finite abelian groups.
This is our first reason for including the appendix.
A second reason is that the argument is constructive and so it can be useful in actually decomposing given bilinear forms 
over finite abelian groups into irreducibles. 
A third reason is that part (b) of theorem \ref{t-wall} as well as lemma \ref{l-rank-two-irr-q} (which we need in our arguments)
are not explicitly stated in \cite{W:QF}. They can probably be extracted from
the arguments in \cite{W:QF} or the other references  \cite{MI:SB}, \cite{VVN:ISBF}.
But this might require some work mainly because each paper has its own and rather
complicated set of notations.
\par
The following lemma, describing the non-degenerate quadratic forms on $(\ZZ/2^r \ZZ)^2$,
is essential to the proof of theorem \ref{t-wall}.
It is stated here because we shall also use it in the computation of some Gauss sums.
It can be proved using Hensel's lemma. A proof is given in appendix \ref{app-A}.
\begin{lemma}
Let $q$ be an irreducible non-degenerate quadratic form on $G = (\ZZ/2 ^r \ZZ)^2$.
Then there exists $A, B, C \in \ZZ$ with $B$ odd such that
$q(x_1, x_2) = 2^{-r} (A x_1^2 + B x_1 x_2 + C x_2^2)$.
If $AC$ is even, then $(G, q) \simeq ((\ZZ/2 ^r \ZZ)^2, x_1 x_2/ 2^r)$.
Otherwise $(G, q) \simeq ((\ZZ/2 ^r \ZZ)^2, (x_1^2 + x_1 x_2 + x_2^2)/ 2^r)$.
\label{l-rank-two-irr-q}
\end{lemma}
\section{Gauss sums and related invariants of a quadratic form}
%
%
\label{section-Gauss-sum}
Let $G$ be a finite abelian group and $q : G \to \QQ/ \ZZ$ be a quadratic form on $G$.
In section \ref{section-introduction}, we defined  the quadratic Gauss sums $\Theta(G,q)$ and the related invariant
$\xi_k(G,q)$, see equations \eqref{e-def-Theta} and \eqref{e-def-xi}.
In this section we shall compute the invariants $\Theta(G,q)$ and 
$\xi_k(G,q) $ for various pairs $(G,q)$. 
One verifies that $\Theta$ is multiplicative, that is, 
\begin{equation*}
\Theta((G_1, q_1) \bot (G_2, q_2) ) = \Theta(G_1,q_1) \Theta(G_2, q_2).
\end{equation*}
In the same sense, $\xi_k$ is also multiplicative.
We start with the following well known result. The proof is omitted.
\begin{theorem}
(a) Let $\chi : G \to \CC^*$ be a character on $G$. Then $\sum_{x \in G} \chi(x) = \abs{G}$ if $\chi = 1$
and $\sum_{x \in G} \chi(x)= 0$ otherwise. 
\par
(b) If $q$ is a non-degenerate quadratic form on $G$, then $\Theta(G, q) \Theta(G, -q) = 1$, in particular,
 $\abs{\Theta(G,q)}^2 = 1$.
\label{l-absolute-value-of-Gauss-sum}
\end{theorem}
The next lemma gives the values of the Gauss sums of irreducible non-degenerate forms.
\begin{lemma}
(a) Let $p$ be an odd prime and $\alpha$ be an integer relatively prime to $p$. Then  
\begin{equation*}
\Theta\bigl( \ZZ/ p^r \ZZ, \alpha(p^r + 1) x^2 / 2p^r) =  \Bigl( \frac{2 \alpha}{p} \Bigr)^r \epsilon_{p^r},
\end{equation*}
where $\Bigl( \frac{2 \alpha}{p} \Bigr)$ denotes the Legendre symbol, and  
$\epsilon_{m} = 1$ if $m \equiv 1 \bmod 4$ and $\epsilon_{m}  = i$ if $m \equiv 3 \bmod 4$.
\par
(b) Let $\alpha$ be an odd integer. Then 
\begin{equation*}
\Theta( \ZZ/ 2^r \ZZ, \alpha x^2/ 2^{r+1}) =  (-1)^{r(\alpha^2 - 1)/8} \be( \alpha/8).
\end{equation*}
\par
(c) Let $\alpha,\beta,\gamma$ be integers with $\beta$ odd. Then
\begin{equation*}
\Theta ( (\ZZ/ 2^r \ZZ)^2 , (\alpha  x_1^2 + \beta x_1 x_2 + \gamma  x_2^2)/ 2^r )
 =  (-1)^{ \alpha \gamma r}.
\end{equation*}
\label{l-gauss-sums-of-irreducible-forms}
\end{lemma}
\begin{proof}
For part (a), see for example \cite{IK:ANT}, page 52. Let $G_r$ and $G'_r$ denote the left hand side of the formulae in part (b)
and part (c) respectively. Then one verifies that $G_r = 2 G_{r - 2} $ and 
$G'_r = 4 G'_{r - 2}$ for $r > 2$. Parts (b) and (c) now follow by induction once the formulae for $r = 1$ and $2$ are verified.
\end{proof}
Since $\Theta$ is multiplicative, one can calculate the Gauss sums of
arbitrary non-degenerate forms by first decomposing the forms into orthogonal direct sum of irreducible forms
and using lemma \ref{l-gauss-sums-of-irreducible-forms}.
We will also need to compute the Gauss sums of some singular forms. This is the purpose of the lemma below.
\begin{lemma} (a) Let $p$ be a prime.  Let $G = (\ZZ/ p^r \ZZ)^n$ and let $q$ be a $p^{-r} \ZZ/ \ZZ$ valued quadratic form on $G$.
Let $0 \leq s \leq r$. Then $p^s q$ induces a quadratic form  on $G/p^{r - s} G$ and 
\begin{equation*}
\Theta( G, p^s q ) = p^{sn/2} \Theta( G/ p^{r - s} G, p^s q ).
\end{equation*}
\par
(b) Let $\alpha$ be an odd integer. Then one has
\begin{equation*}
\Theta\bigl( \ZZ/2^r \ZZ, 2^{s} \cdot \tfrac{\alpha x^2}{2^{r+1}} \bigr) 
 =\begin{cases}
 2^{s/2} (-1)^{(r-s)(\alpha^2 - 1)/8} \be( \alpha/8) 
  & \text{\; if \;} 0 \leq s < r \\
  0 & \text{\; if \;} s = r \\
  2^{r/2} & \text{\; if \;} s > r.
  \end{cases}
 \end{equation*}
\label{l-theta-of-scaled-form}
\end{lemma}
\begin{proof}
(a) If $x \equiv x' \bmod p^{r - s} G$, then $p^s q(x) = p^s q(x')$ since $q$ and $\partial q$ takes values in
$p^{-r} \ZZ/ \ZZ$. So $p^s q(x)$ induces a form on $G/p^{r - s}G$.
One has
\begin{align*}
{\abs{G}}^{1/2} \Theta(G, p^s q) 
= \sum_{x \in G} \be(p^s q(x))  
&= \abs{p^{r - s} G}  \sum_{y \in G/p^{r - s} G} \be(p^s q(y)) \\
&= \abs{p^{r - s} G} \abs{G / p^{r-s} G}^{1/2} \Theta( G/p^{r - s}G, p^s q).
\end{align*}
Part (a) follows since $\abs{ p^{r - s} G} = p^{sn}$. 
\par
(b) First suppose $r - s \geq 1$. Note that if $y \equiv x \bmod 2^{r - s}$, then
$ \tfrac{\alpha y^2}{ 2^{r - s + 1}} \equiv \tfrac{\alpha x^2}{2^{r - s + 1}} \bmod \ZZ$.
So
\begin{equation*}
2^{r/2} \Theta\bigl( \ZZ/2^r \ZZ, 2^{s} \cdot \tfrac{\alpha x^2}{2^{r+1}} \bigr)= 
\sum_{x = 0}^{2^r - 1} \be( \tfrac{ \alpha x^2}{2^{r - s + 1} })= 
2^s  \sum_{x = 0}^{2^{r - s} - 1} \be( \tfrac{ \alpha x^2}{2^{r - s + 1} })
= 2^{(r+s)/2} \Theta \bigl( \ZZ/2^{r-s} \ZZ, \tfrac{\alpha x^2}{2^{r - s + 1} } \bigr).
\end{equation*}
Part (b) now follows from lemma \ref{l-gauss-sums-of-irreducible-forms} for $0 \leq s < r$.
Now let $s = r$. Note that if $y \equiv x \bmod 2$, then $\tfrac{ \alpha x^2}{2 } \equiv \tfrac{ \alpha y^2}{2} \bmod \ZZ$. So
\begin{equation*}
2^{r/2} \Theta\bigl( \ZZ/2^r \ZZ, 2^{s} \cdot \tfrac{\alpha x^2}{2^{r+1}} \bigr)
=  \sum_{x = 0}^{2^r - 1} \be( \tfrac{ \alpha x^2}{2 })
= 2^{r-1}  \sum_{x = 0}^1  \be( \tfrac{ \alpha x^2}{2 })
=0.
\end{equation*}
For $s > r$, the quadratic form we have is identically equal to $0$, so the result is obvious.
 \end{proof}
\begin{lemma}
Let $p$ be an odd prime and let both $r$ and $k$ be positive integers. Let $q_1$ and $q_2$ be the two non-isometric 
non-degenerate quadratic forms on $G = \ZZ/p^r \ZZ$. Then
\begin{equation*}
\xi_k(G,q_1)
= (-1)^{ \epsilon_{p,r}^k } \xi_k(G,q_2)
\end{equation*}
where $\epsilon_{p,r}^k = r(k+1) - \min\{r, v_p(k)\}$.
\label{l-q1q2s}
 \end{lemma}
\begin{proof} 
There are only two distinct non-generate quadratic forms on $G$, see table \ref{table-1}.
Without loss of generality, we may thus assume  that $q_j(x) = u_j (p^r + 1)x^2/2p^r$ for $j = 1, 2$, where $u_1 = 1$ and $u_2 = u_p$
is a quadratic non-residue modulo $p$.  Lemma \ref{l-gauss-sums-of-irreducible-forms} (a) implies 
$\Theta(G,q_1) = (-1)^r\Theta(G,q_2)$.
If $v_p(k) > r$ the lemma holds by the fact that $\Theta(G,-kq) = \sqrt{\abs{G}}$. 
\par
Now assume $0 \leq v_p(k) \leq r$.
Write $s = v_p(k)$ and $-k  = p^s a$ with $a \in \ZZ$ relatively prime to $s$. 
Then $\Theta(G, -k q_j)$ is equal to
\begin{align*}
\Theta(G, p^s a q_j) 
&= p^{s/2}\Theta(\ZZ/p^{r-s}\ZZ, p^s a u_j (p^r + 1) x^2/ 2 p^r)  \\
&= p^{s/2} \Theta \left(\ZZ/ p^{ r - s} \ZZ, (p^{r-s} + 1) a u_j x^2/ 2 p^{r - s}\right).
\end{align*}
The first equality follows from lemma \ref{l-theta-of-scaled-form}(a). For the second, we need to observe that
the quadratic forms
$(p^{r-s} + 1) \alpha x^2/ 2 p^{r - s}$
and $(p^{r} + 1) \alpha x^2/ 2 p^{r - s}$ are identical on
$\ZZ/ p^{r - s} \ZZ$. 
From lemma \ref{l-gauss-sums-of-irreducible-forms} (a) we have
\begin{equation*}
\Theta \left(\ZZ/ p^{ r - s} \ZZ, (p^{r-s} + 1) a u_p x^2/ 2 p^{r - s}\right) = (-1)^{r-s}\Theta \left(\ZZ/ p^{ r - s} \ZZ, (p^{r-s} + 1) a x^2/ 2 p^{r - s}\right)
\end{equation*}
which implies $\Theta(G,-kq_2) = (-1)^{r - v_p(k)}\Theta(G,-kq_1)$. The lemma follows, once we recall that
$\Theta(G,q_1) = (-1)^r\Theta(G,q_2)$.
 \end{proof}
\par
Next, we shall introduce an invariant $\sigma_k(b)$ of a discriminant form $(G,b)$ defined in \cite{KK:LP} and  in lemma \ref{l-sigma-and-Theta} compare it to our Gauss sums (Note: discriminant forms are called linking pairs in \cite{KK:LP}).
\newline
\par
{\bf Definition.} For the convenience of the reader we shall recall some of the definitions from \cite{KK:LP} and \cite{W:QF}. 
Let $G$ be a finite abelian group. Let
\begin{equation*}
G[n] = \lbrace x \in G \colon n x = 0 \rbrace 
\end{equation*}
denote the $n$-torsion subgroup of $G$.
Let $p$ be a prime. Then $G_{(p)} = \cup_{n} G[p^n]$ is the $p$-Sylow subgroup of $G$.
For $k \geq 1$, define
\begin{equation*}
\tilde{G}_p^k = G[p^k] / ( G[p^{k-1}] + p G[p^{k+1}] ).
\end{equation*}
Take a decomposition of $G$ into a direct sum of cyclic groups of prime power order. 
If such a decomposition has $n$ factors isomorphic to $\ZZ/ p^k \ZZ$, Then 
$\tilde{G}_p^k$ is an elementary abelian $p$-group of rank $n$. 
Let $b$ be a non-degenerate symmetric bilinear form on $G$. Then 
\begin{equation*}
\tilde{b}_p^k ([x],[y]) = p^{k - 1} b(x,y)
\end{equation*}
defines a non-degenerate bilinear form on $\tilde{G}_p^k$. Here $x$ and $y$ denote any two elements
of $G[p^k]$ representing $[x], [y] \in \tilde{G}_p^k$ respectively.
\par
Let $c^k(b)$ be the characteristic element (also called parity element) of the $\mathbb{F}_2$-quadratic space 
$(\tilde{G}_2^k, \tilde{b}_2^k)$. Explicitly, $c^k(b)$ is the unique element of  $\tilde{G}_2^k$
such that $\tilde{b}_2^k( x, x) = \tilde{b}_2^k( x, c^k(b))$ for all $x\in \tilde{G}_2^k$.
In other words, $c^k(b)$ is represented by any $c \in G[2^k]$ that satisfies
\begin{equation*}
2^{k-1} b(x,x) = 2^{k-1} b(x, c) \text{ \; for all \;} x \in G[2^k]. 
\end{equation*}
Note that both sides of the above equality can only take the values $0$ or $1/2$. Also observe that
the characteristic element $c^k(b)$ is zero if and only if $b(x,x) \in 2^{ 1 - k} \ZZ/ \ZZ$ for all $x \in G[2^k]$.
\par
The invariant $\sigma_k(b)$ takes value in $(\ZZ/ 8 \ZZ) \cup \lbrace \infty \rbrace$, which is made into a semigroup
by defining $\infty + \infty = n + \infty = \infty$ for $n \in \ZZ/ 8 \ZZ$. If $c^k(b) \neq 0$, then  $\sigma_n(b) = \infty$
by definition. If $c^k(b) = 0$, then one checks that 
\begin{equation*}
q_k(x) =  2^{k-1} b(x,x)
\end{equation*}
 induces a well defined quadratic form
on $G_{(2)}/G[2^k]$ and, following \cite{KK:LP}, we  can define $\sigma_k(b)$ by
\begin{equation*}
\abs{ G_{(2)}/G[2^k] }^{1/2} \Theta(G_{(2)}/G[2^k], q_k) = C \be( \sigma_k(b)/8),
\end{equation*}
where $C$ is the absolute value of the left hand side of the equation (see \cite{KK:LP}, section 2); we shall soon see that $C \neq 0$.
If $x, y \in G_{(2)}$ represents $[x], [y] \in G_{(2)}/G[2^k]$, then
$\partial q_k([x],[y]) = 2^k b(x,y)$. Suppose $[x] \in G_{(2)}/G[2^k]$ such that
$\partial q_k([x],[y] ) = 0$ for all $[y] \in G_{(2)}/G[2^k]$. Let $x \in G_{(2)}$ be a representative for $[x]$.
Then $2^kb(x,y) = 0$ for all $y \in G_{(2)}$. Since $b$ is non-degenerate, it follows that 
$2^k x = 0$, so $[x] = 0$ in $G_{(2)}/G[2^k]$. So we have argued that if $c^k(b) = 0$, then
$q_k(x)$ is a non-degenerate form on $G_{(2)}/G[2^k]$. Hence lemma \ref{l-absolute-value-of-Gauss-sum}(b)
gives $C =  \abs{G_{(2)}/G[2^k]}^{1/2}$. So $\sigma_k(b)$ is in fact given by the simpler formula
\begin{equation}
\Theta(G_{(2)} \slash G[2^k], q_k) = \be(\sigma_k(b) \slash 8).
\label{eqn-Theta-and-sigma}
\end{equation}
The following theorem is the reason for our interest in the invariant $\sigma_k(b)$ and it follows from theorem 4.1 of \cite{KK:LP}.
\begin{theorem}[\cite{KK:LP}] Let $G$ be a finite abelian $2$-group and let $b$ and $b'$ be two
non-degenerate symmetric bilinear forms on $G$. Then $(G, b) \simeq (G, b')$ if and only if 
$\sigma_k(b) \simeq \sigma_k(b')$ for all $k \geq 1$.
\label{t-KK-quote}
\end{theorem}
%
%
\par
{\bf Definition.} It will be convenient for us to work with the invariant
\begin{equation}
\varsigma_k(b) = \be( \sigma_k(b)/8)
\label{th-def-varsigma}
\end{equation}
rather than $\sigma_k(b)$.
If $\sigma_k(b) = \infty$, then we
define  $\varsigma_k(b) = 0$. So $\varsigma_k$ takes values in the multiplicative semigroup
$\mu_8 \cup \lbrace 0 \rbrace$ where $\mu_8$ is the group of $8$-th roots of unity. 
From corollary 2.2 of \cite{KK:LP}, it follows that if $(G, b) = (G_1, b_1) \bot (G_2, b_2)$, the $\varsigma_k(b)= \varsigma_k(b_1) \varsigma_k(b_2)$.
In other words, $\varsigma_k$ is multiplicative, just like the Gauss sums or the invariant $\xi_k$.
The multiplicativity of $\varsigma_k(b)$ also follows from the next lemma.
\begin{lemma}
Let $G$ be a finite abelian $2$-group and let $b$ be a non-degenerate symmetric bilinear form on $G$. 
Let  $k \geq 1$. Then
\begin{equation*}
\Theta( G, 2^{k-1} b(x,x)) = \abs{G[2^k]}^{1/2} \varsigma_k(b). 
\end{equation*}
Let $q$ is a non-degenerate quadratic form on $G$. Then with $b = \partial q$, the above equation yields
\begin{equation}
\varsigma_k( \partial q) = \abs{ G[2^k] }^{-1/2} \Theta( G, 2^k q).
\label{eq-varsigma_n-partial-q}
\end{equation}
\label{l-sigma-and-Theta}
\end{lemma}
\begin{proof} Let $q_k(x) = 2^{k-1} b(x,x)$.
Let $w$ vary over a set of coset representatives of $G/G[2^k]$ and $y$ vary over $G[2^k]$. Then
\begin{equation}
\abs{G}^{1/2} \Theta( G, q_k) 
= \sum_{w, y} \be( q_k(w+y)  )  
= \sum_{w} \be( q_k(w)) \sum_y \be(2^{k-1} b(y,c^k(b)) ).
\label{e-tgs} 
\end{equation}
The second equality follows since $2^k b( w, y) = 0$ and $2^{k-1} b(y,y) = 2^{k-1} b(y, c^k(b))$. 
If $c^k(b) \neq 0$, then $y \mapsto \be(2^{k-1} b(y,c^k(b)) )$ is a non-trivial character on $G[2^k]$, so 
the inner sum in \eqref{e-tgs} is zero, hence $\Theta( G, 2^{k-1} b(x,x))  = 0$. 
Now suppose $c^k(b) = 0$. Then we find that
$2^{k-1}b(w,w) = 2^{k-1} b(w',w')$ if $w \equiv w' \bmod G[2^k]$. Thus,
$( w \mapsto q_k(w))$ induces a quadratic form on $G/G[2^k]$. From equation \eqref{e-tgs}, we get
\begin{equation*}
\abs{G}^{1/2} \Theta(G, q_k) 
= \abs{G[2^k]} \sum_{w \in G \slash G[2^k]} \be(q_k(w))
= \abs{G[2^k]} \sqrt{ \abs{ G/G[2^k] }} \Theta(G \slash G[2^k], q_k).
\end{equation*}
The lemma follows from equation \eqref{eqn-Theta-and-sigma}.
 \end{proof}
\begin{lemma}
Let $(G, q)$ be an irreducible metric $2$-group with $\exp(G) = 2^r$ (see table \ref{table-1}). 
Let  $\beta$ be an odd integer and $n \geq 1$.  Then
\begin{equation}
\varsigma_n( \partial q)^{ \beta 2^n} =
 \begin{cases} 0 & \text{ \; if \; } n = r \text{\;and \;} \op{rk}(G) = 1, \\
   (-1)^{\op{rk}(G) \delta_{n,2}  \delta_{r,1}  } \Theta(G,q)^{\beta 2^n} & \text{\; otherwise}.
   \end{cases}
\label{eq-pow-of-theta}
\end{equation}
where $\delta_{i,j}$ is the Kronecker delta, and 
\begin{equation}
\Theta(G, \beta 2^n q) 
= \abs{G[2^n]}^{1/2} (-1)^{\op{rk}(G) \max \lbrace r - n , 0 \rbrace (\beta^2 - 1)/8 } 
\varsigma_n(\partial q)^{\beta}.
\label{eq-theta-of-scaled-q}
\end{equation}
\label{l-theta-of-pieces-of-Tk}
\end{lemma}
\begin{proof}
Then $\op{rk}(G) = 1$ or $2$. We treat these cases separately
First suppose $G$ has rank $1$, that is,
$(G, q) \simeq ( \ZZ/ 2^r \ZZ, \alpha x^2/ 2^{r + 1} )$ where $\alpha \in \lbrace \pm 1, \pm 5 \rbrace$.
Then from lemma \ref{l-gauss-sums-of-irreducible-forms}(b), we find that 
$\Theta(G, q) = \pm \be( \alpha/8)$. Since $n \geq 1$, we have
\begin{equation}
\Theta(G, q)^{ \beta 2^n} = \be( \alpha/8)^{\beta 2^n}.
\label{eq-pow-of-theta-explicit}
\end{equation}
Now we split the argument in three cases.
\par
{\bf Case 1:} $n > r$. Then $\Theta(G, 2^n \beta q) = \abs{G}^{1/2} = \abs{G[2^n]}^{1/2}$
and so equation \eqref{eq-varsigma_n-partial-q} implies $\varsigma_n(\partial q) = 1$.
This verifies equation \eqref{eq-theta-of-scaled-q}.
From equation \eqref{eq-pow-of-theta-explicit} we obtain
$\Theta(G, q)^{ \beta 2^n} = \be( \alpha/8)^{ \beta 2^n} = (-1)^{ \delta_{n , 2} \delta_{r,1}}$.
This verifies equation \eqref{eq-pow-of-theta}.
\par 
{\bf Case 2:} $ n  =  r$. Lemma \ref{l-theta-of-scaled-form} (b) implies that 
$\Theta( G, 2^n \beta q) = 0$.
 From equation \eqref{eq-varsigma_n-partial-q} we get 
 $\varsigma_n(\partial q) = \abs{G[2^n]}^{-1/2}  \Theta( G, 2^n q) =   0$ too.
 This verifies equations \eqref{eq-pow-of-theta} and \eqref{eq-theta-of-scaled-q} in this case.
\par
{\bf Case 3:} $1 \leq n < r$.
From equation \eqref{eq-varsigma_n-partial-q} and lemma \ref{l-theta-of-scaled-form} (b), we have, 
\begin{equation*}
\varsigma_n(\partial q) 
= \abs{G[2^n]}^{-1/2} \Theta(G, 2^n q)
= 2^{-n/2} \Theta( \ZZ/2^r \ZZ, 2^n \tfrac{ \alpha x^2}{2^{r + 1}} )
=  (-1)^{(r - n)(\alpha^2 - 1)/8}  \be(\tfrac{ \alpha}{8}).
\end{equation*}
Since $n \geq 1$, using equation \eqref{eq-pow-of-theta-explicit} we obtain 
$\varsigma_n(\partial q)^{\beta 2^n} = \be(\tfrac{ \alpha}{8})^{ \beta 2^n} = \Theta( G, q)^{\beta 2^n}$ 
which verifies equation \eqref{eq-pow-of-theta}. 
To verify the expression for $\Theta( G, \beta 2^n q)$ we compute as follows: 
\begin{align*}
\Theta( G, 2^n \beta q) 
= \Theta( \ZZ/2^r \ZZ, 2^n \tfrac{\beta \alpha x^2}{2^{r+1}} ) 
& = 2^{n/2}(-1)^{(r- n)(\alpha^2 \beta^2 - 1)/8} \be(\tfrac{\beta \alpha}{8}) \\
& =  2^{n/2}(-1)^{(r- n)(\beta^2 - 1)/8}\bigl( (-1)^{(r - n)(\alpha^2 - 1)/8}  \be(\tfrac{ \alpha}{8}) \bigr)^{\beta} \\
& = 2^{n/2}(-1)^{(r- n)(\beta^2 - 1)/8}\varsigma_n(\partial q)^{\beta},
\end{align*}
where the third equality follows from the fact that for odd integers $\beta, \alpha$
\begin{equation}
	(\alpha^2\beta^2 - 1) - (\beta^2 - 1) - \beta(\alpha^2 - 1) = \beta(\beta - 1)(\alpha^2 - 1) \equiv 0 \bmod 16
	\label{eqn:alpha-beta-16}
\end{equation}
  This verifies equation \eqref{eq-theta-of-scaled-q} and finishes the argument  in the case $\op{rk}(G) = 1$.
  \par
  Now assume $\op{rk}(G) = 2$.
 If $n < r$, then equation \eqref{eq-varsigma_n-partial-q}, lemma \ref{l-theta-of-scaled-form}(a) and 
 \ref{l-gauss-sums-of-irreducible-forms}(c) gives us
$\varsigma_n(\partial q) = \pm 1$ (or else see corollary 2.2 of \cite{KK:LP}).  If $n \geq r$, then from equation \eqref{eqn-Theta-and-sigma} we obtain, 
$\varsigma_n(\partial q) = \Theta(G \slash G[2^n], 2^nq)$. 
Since $G[2^n] = G$, the Gauss sum is equal to $1$ and thus $\varsigma_n(\partial q) = 1$. 
Thus, in any case, we find that $\varsigma_n( \partial q) = \pm 1$.
Lemma \ref{l-gauss-sums-of-irreducible-forms}(c) tells us that $\Theta( G, q) = \pm 1$ as well. 
Now equation \eqref{eq-pow-of-theta} follows since $n \geq 1$.
\par
Since $\varsigma_n( \partial q) = \pm 1$, the right hands side of equation \eqref{eq-theta-of-scaled-q} becomes
\begin{equation*}
\abs{G[2^n]}^{1/2} \varsigma_n(\partial q).
\end{equation*}
 Since $G$ is of type $E_{2^r}$ or $F_{2^r}$, lemma \ref{l-rank-two-irr-q} implies $(G, \beta q) \simeq (G, q)$. So
 $(G, 2^n \beta q) \simeq (G, 2^n q)$. So equation \eqref{eq-theta-of-scaled-q} follows immediately from
equation \eqref{eq-varsigma_n-partial-q}. 
\end{proof}
\begin{lemma}
Let $(G,q)$ be a metric $2$-group. Let $n \geq 1$ and $\beta$ be an odd positive integer.  Let $\varsigma_n(\partial q)$ be the invariant introduced in equation \eqref{th-def-varsigma}.  Then
\begin{equation*}
\xi_{2^n \beta} (G, q)
= (-1)^{\Gamma_{G,\beta,n} }  \abs{G[2^n]}^{1/2} \varsigma_n(\partial q)^{(2^n-1)\beta}
\end{equation*}
where $\Gamma_{G,\beta,n}$ is an integer dependent on $G$,$\beta$, $n$ and independent of $q$.
More precisely, if we write $G \simeq \oplus_{r = 1}^{\infty} (\ZZ/2^r \ZZ)^{N_r}$, then
$\Gamma_{G, \beta, n} = \delta_{n,2} N_1 + \sum_{r = 1}^{\infty} N_r \op{max} \lbrace r - n , 0 \rbrace (\beta^2 - 1)/8$.
\label{l-two-power-ind}
\end{lemma}
\begin{proof}
Observe that both sides of the equation we want to prove are multiplicative invariants of a metric group.
Since any metric group $(G,q)$ can be decomposed into irreducibles by theorem \ref{t-wall}, it suffices to prove the
equation when $(G,q)$ is an irreducible metric group.
Assume $(G,q)$ is an irreducible metric group of exponent $2^r$; 
the the possibilities for these given in table \ref{table-1}. Note that $G$ is isomorphic to $ (\ZZ/ 2^{r} \ZZ)$ or 
$(\ZZ/ 2^{r} \ZZ)^2$ and $N_j = \delta_{j,r} \op{rk}(G)$. So the equation we want to prove becomes
\begin{equation*}
\Theta(G,q)^{\beta 2^n} \Theta(G,-\beta 2^n q) 
= (-1)^{\op{rk}(G) \delta_{n,2}  \delta_{1, r}  +  \op{rk}(G)  \op{max} \lbrace r - n , 0 \rbrace (\beta^2 - 1)/8}  
\abs{G[2^n]}^{1/2} \varsigma_n(\partial q)^{(2^n-1)\beta}.
\end{equation*}
This equation follows directly from lemma \ref{l-theta-of-pieces-of-Tk}.
\end{proof}
%
%
%
\section{Indicator of Tambara-Yamagami categories as Gauss sums}
\label{section-indicator}
%
%
Let $G$ be a finite abelian group.
A function $\chi: G \times G \to \CC^*$ is called a {\it symmetric bicharacter} on $G$ if $\chi(x, \cdot)$ and $\chi(\cdot, x)$
are characters on $G$ and $\chi(x,y) = \chi(y,x)$ for each $x,y \in G$.
A symmetric bilinear form $b$ on $G$ determines a symmetric bicharacter $\chi: G \times G \to \CC^*$ given
by $\chi(x, y) = \be(-b(x,y))$ (The minus sign in front of $b$ is for consistency with notation in \cite{S:FSI}).
This sets up a natural correspondence between bilinear forms
and bicharacters. We say $\chi$ is non-degenerate if $b$ is and vice versa. 
\par
Let $G$ be a finite abelian group, $\chi$ be a non-degenerate symmetric bicharacter on $G$ and $\tau$ be a square
root of $\abs{G}^{-1}$. Let $b$ be the bilinear form on $G$ given by $\chi(x,y) = \be(-b(x,y))$.
Given any triple $(G, \chi, \tau)$, there exists a spherical fusion category $\cC$,
called the Tambara-Yamagami category or $\op{TY}$-category for short. We shall denote this category by
$\op{TY}(G, \chi, \tau)$ or by $\op{TY}(G, b, \tau)$.
The simple objects of $\cC$ are $G \cup \lbrace m \rbrace$.
We shall write $m = m_{\cC}$ if there is a chance of confusion.
The associativity constraint in $\op{TY}(G, \chi, \tau)$ is dictated by the bicharacter $\chi$
and $\op{sign}(\tau)$. See \cite{TY:TC} or \cite{S:FSI} for more details on the $\op{TY}$-categories. 
{\bf Caution:} The abelian groups in \cite{S:FSI} are multiplicative, while for our purpose
it is convenient to write the group $G$ additively.
\par
For each simple object $x$ of a spherical fusion category and each integer $k \geq 1$, one can associate
a complex number $\nu_{k}(x)$, introduced in \cite{NS:FS},
called the $k$-th Frobenius-Schur indicator of $x$.
The lemma below tells us the indicators of the simple objects of a TY-category. 
This is an easy translation of results in \cite{S:FSI}.
Our main observation is noting that the indicators of the object $m_{\cC}$ can be 
expressed in terms of certain Gauss sums. 
%
%
\begin{lemma}
Let $\cC = \op{TY}(G, \chi, \tau)$ be a $\op{TY}$-category. From theorem 3.2 of \cite{S:FSI} we know that
$\nu_k(x) = \delta_{x^k, 1}$ for $x \in G$. 
Let $b$ be the bilinear form on $G$ given by $\chi(x,y) = \be(-b(x,y))$. Let $q$ be any quadratic form such that
$\partial q = b$. Then for all $k \geq 1$, one has $\nu_{2k - 1} (m_{\cC}) = 0$ and
\begin{equation*}
\nu_{2k}( m_{\cC} ) = \op{sign}(\tau)^k \Theta(G,q)^k \Theta(G, -kq) = \op{sign}(\tau)^k \xi_k(G, q) ,
\end{equation*}
and this value does not depend on the choice of $q$. 
\label{l-nu-and-Theta}
\end{lemma}
\begin{proof} 
From theorem 3.3 of \cite{S:FSI}, we know that $\nu_{2 k  - 1}(m) = 0$.  Let
\begin{equation*}
C(\chi) = \lbrace \varphi  : G \to \CC \; : \; \varphi(x)\varphi(y)\varphi(x+y)^{-1} = \chi(x,y) \text{ for } x,y \in G \rbrace.
\end{equation*}
From the proof of theorem 3.3 of \cite{S:FSI} we have
\begin{equation}
\nu_{2 k}(m_{\cC}) 
= \frac{1}{\abs{G}} \sum_{\varphi \in C(\chi)} \left(\tau\sum_{x \in G} \varphi(x)\right)^k \sqrt{\abs{G}}.
\label{eq-from-Shimizu}
\end{equation}
By definition $\be(q) \in C(\chi)$. 
One checks that $G$ acts simply transitively on $C(\chi)$ by $a \cdot \varphi(x) = \varphi(x)\chi(a,x)^{-1}$.
So $C(\chi) = \lbrace \varphi_a \colon a \in G \rbrace$  where $\varphi_a(x) =  \be(q(x)) \chi(a,x)^{-1} $.
One has
\begin{align*}
\tau \sum_{x \in G} \varphi_a(x) &= \frac{\op{sign}(\tau)}{\sqrt{\abs{G}}}\sum_{x \in G} \be(q(x) + b(a, x) + q(a) - q(a))\\
&= \frac{\op{sign}(\tau)\be(-q(a))}{\sqrt{\abs{G}}} \sum_{x \in G} \be(q(x+a))\\
&= \op{sign}(\tau)\be(- q(a)) \Theta(G,q).
\end{align*}
From equation \eqref{eq-from-Shimizu}, it follows that 
\begin{align*}
\nu_{2 k}(m_{\cC}) 
&= \frac{\op{sign}(\tau)^k}{\sqrt{\abs{G}}} \sum_{a \in G} \be(- k q(a)) \Theta(G,q)^k\\
&= \op{sign}(\tau)^k\Theta(G,q)^k \Theta(G, -kq).
\end{align*}
 To complete the proof observe that the expression on the right hand side of \eqref{eq-from-Shimizu} only depends on $\chi$ and is independent of the choice of $q$.
 \end{proof}
We shall need the following result from \cite{S:FSI}. 
\begin{lemma}[\cite{S:FSI}, theorem 3.5]
Let $\cC = \op{TY}(G, b, \tau)$ be a $\op{TY}$-category.  Let $q$ be a quadratic form
such that $\partial q = b$. Then $\nu_{2k }(m) = \abs{G[k]}^{1/2} \psi$
where $\psi \in \mu_8 \cup \lbrace 0 \rbrace$ (recall: $\mu_8$ denotes the set of $8$-th roots of unity).
One has $\psi = 0$ if and only if there exists $a \in G[k]$ such that
$k q(a) \neq 0$. 
\label{l-abs-nu}
\end{lemma}
%
%
\par
{\bf Remark.} We should mention that from the values of the Gauss sums given in the previous section and
the decomposition of $(G,q)$ into irreducibles  we can show that $\xi_k(G,q) = 0$ if and only if 
$(G,q)$ contains an irreducible component which 
equals $A_{2^r}$, $B_{2^r}$, $C_{2^r}$, or $D_{2^r}$ 
where $r = v_2(k)$  for some even $k$
and this yields another proof of \ref{l-abs-nu}.
\par
Let $(G,q)$ be a pre-metric group.
The invariant $\xi_k(G,q)$ can itself be expressed as a Gauss sum as follows. Let $\mathcal{F}_k(G,q)$
denote the pre-metric group given by the abelian group
$\lbrace (g_1, \dotsb, g_k) \in G^k \colon \sum_{j} g_j = 0 \rbrace$ with the quadratic form
$q(g_1, \dotsb, g_k) = \sum_{j} q(g_j)$. Then one can show that $\xi_k(G,q) = \mathcal{F}_k(G,q)$.
In view of this formula, the appearance of the $8$-th root of unity $\psi$ in the above lemma
becomes a consequence of Milgram's formula. 
%
%
%
\section{Tambara-Yamagami categories are determined by the higher Frobenius Schur-indicators}
\label{section-TY}
%
%
In this section we shall prove theorem \ref{t-indicators-determine-b}.
Let $\cC = \op{TY}(G,\chi, \tau)$  be a $\op{TY}$-category.
We shall show that
the Frobenius-Schur indicators of the simple objects of $\cC$ determine the triple $(G, \chi, \tau)$.
So the indicators can distinguish between any two $\op{TY}$-categories.
Most of the work goes into showing that the indicators $\nu_k(m_{\cC})$ determine the
bicharacter $\chi$.
Let $q$ be a quadratic form on $G$ such that $\chi(x,y) = \be( - \partial q(x,y))$.
Then lemma \ref{l-nu-and-Theta} gives $\nu_k(m_{\cC}) = \op{sgn}(\tau)^k \xi_k(G,q)$
where $\xi_k(G,q)$ is a product of quadratic Gauss sums.
Based on computations in section \ref{section-Gauss-sum},
we shall argue that the invariants $\xi_k(G,q)$ determine the bicharacter $\chi$.
We need a couple of lemmas before proving theorem \ref{t-indicators-determine-b}.
The lemmas let us handle special cases.
\begin{lemma}
Let $G$ be an abelian group of odd order. 
Let $b_1$ and $b_2$ be two non-isometric non-degenerate symmetric bilinear forms on $G$.
Let $q_1$ and $q_2$ be quadratic forms such that $\partial q_j= b_j$ for $j=1,2$. 
Then either there exists an odd positive integer $k$ such that
$\xi_k(G , q_1) \neq \xi_k(G , q_2)$, or else, for each natural number $\gamma $, there exists a
positive integer $k$ with $v_2(k) = \gamma$ and $\xi_k(G , q_1) \neq \xi_k(G , q_2)$.
\label{l-dop}
\end{lemma}
\begin{proof} 
Fix a non-square $u_p$ modulo $p$ for each odd prime $p$. Recall from table 1
\begin{equation*}
A_{p^r} = \Bigl(\ZZ/p^r \ZZ, q(x)  = \tfrac{2^{-1} x^2}{p^{r}} \Bigr)
\text{\; and \;} 
B_{p^r} = \Bigl(\ZZ/p^r \ZZ, q(x)  = \tfrac{2^{-1} u_p x^2}{p^{r}} \Bigr).
\end{equation*}
We will also use the notation
\begin{equation*}
n \cdot A_{p^r} = \Bigl(\ZZ/p^r \ZZ, q(x)  = \tfrac{2^{-1} nx^2}{p^{r}} \Bigr)
\text{\; and \;} 
n \cdot B_{p^r} = \Bigl(\ZZ/p^r \ZZ, q(x)  = \tfrac{2^{-1} u_p n x^2}{p^{r}} \Bigr).
\end{equation*}
for $n \in \ZZ$.  Write $G \simeq \oplus_{p,r}  (\ZZ/ p^r \ZZ)^{N_{p,r}}$
where $p$ ranges over odd primes and $r \geq 1$. Since 
$A_{p^r} \bot A_{p^r} \simeq B_{p^r} \bot B_{p^r}$ (see \cite{W:QF}, theorem 4), 
the metric group $(G,q_j)$ is an orthogonal direct sum, over all $(p,r)$ such that $N_{p,r} \neq 0$, of
the homogeneous metric groups
\begin{equation*}
 A_{p^r}^{N_{p,r} - 1} \bot C_{p,r}^j, 
\end{equation*}
where $C_{p,r}^j$ is either $A_{p^r}$ or $B_{p^r}$.
Since $\xi_k$ is multiplicative, we have
\begin{equation}
\xi_k(G,q_j) = \prod_{p,r : N_{p,r} \neq 0}   \xi_k( A_{p^r})^{N_{p,r} - 1}  \xi_k( C_{p,r}^j)
\label{eq-Theta-long-prod}
\end{equation}
 Let 
\begin{equation*}
\cA = \lbrace (p,r) \colon N_{p,r} \neq 0, C_{p,r}^1 \neq C_{p,r}^2 \rbrace
\text{\; and \;}
\cA_{\max} = \lbrace (p,r) \in \cA \colon (p, r') \notin \cA \text{\; for all \;} r' > r \rbrace.
\end{equation*}
If $(p,r) \notin \cA$, then the $(p,r)$-th term in the product in equation \eqref{eq-Theta-long-prod} 
is the same for $j = 1, 2$. If $(p,r) \in \cA$, then
the $(p,r)$-th terms differs by a factor $(-1)^{\epsilon_{p,r}^k}$ given in lemma \ref{l-q1q2s}.
It follows that
 \begin{equation*}
 \xi_k(G, q_1)  = (-1)^{\Lambda}  \xi_k(G, q_2)
 \text{\; where \;\;} 
  \Lambda = \sum_{(p,r) \in \cA} \epsilon_{p,r}^k.
  \end{equation*}
{\bf Case 1: } If there is a prime $p$ such that $(p,1) \in \cA_{\max}$ then choose such a prime $p_0$ and let
$k = p_0$.  We find
\begin{equation*}
\sum_{r: (p_0,r) \in \cA} \epsilon_{p_0,r}^k = \epsilon_{p_0,1}^k = 1(k+1)  - \min\lbrace 1, v_{p_0}(k) \rbrace
= p_0  \equiv 1 \bmod 2.
\end{equation*}
For all prime $(p, r) \in \cA$ such that $p \neq p_0$ we have
$\epsilon_{p,r}^k = r(p_0 +1) \equiv 0 \bmod 2$.
It follows that 
$\Lambda \equiv 1 \bmod 2$, so 
$ \xi_k(G, q_1) \neq \xi_k(G, q_2)$.
\par
{\bf Case 2: } Otherwise, choose $(p_0,r_0) \in \cA_{\max}$ such that $r_0 > 1$. Choose any $\gamma \geq 1$
and let 
\begin{equation*}
k = 2^{\gamma} p_0^{-1} \prod_{(p,r) \in \cA_{\max} } p^{r}.
\end{equation*}
Note that $k$ is an integer with $v_2(k) = \gamma$ and $v_{p_0}(k) = r_0 - 1$. 
One has
\begin{equation*}
\epsilon_{p_0, r_0}^k =r_0( k + 1) - \min\lbrace r_0, v_{p_0}(k) \rbrace 
 \equiv r_0 - (r_0 - 1) = 1 \bmod 2.
\end{equation*}
If $r < r_0$, then $r \leq v_{p_0} (k)$, so 
$\epsilon_{p_0,r}^k = r( k - 1) - r \equiv 0 \bmod 2$. 
Finally if $p \neq p_0$, then $(p, r) \in \cA$ implies $r \leq v_p(k)$
by our choice of $k$, so $\epsilon_{p,r}^k = r( k + 1) - r \equiv 0 \bmod 2$. 
Again, $\Lambda \equiv 1 \bmod 2$, so 
$\xi_k(G,q_1) \neq \xi_k(G,q_2)$.
 \end{proof}
\begin{lemma}
Let $b$ and $b'$ be two non-degenerate symmetric bilinear forms on a finite abelian $2$-group $G$.
Let $q$ and $q'$ be quadratic forms such that $\partial q = b$ and $\partial q'  = b'$.
Let $k$ be a positive integer such that 
$v_2(k) = 0$ or $v_2(k) > \max\{2, v_2( \exp(G)) \}$. Then 
$\xi_k(G,q) = \xi_k(G,q')$.
\label{l-even-parts-agree}
\end{lemma}
\begin{proof} 
By structure theorem of finite abelian groups and by theorem \ref{t-wall}, we can decompose $G$
and $(G,q)$ respectively as
\begin{equation*}
G \simeq \oplus_{r=1}^{\infty} (\ZZ/2^r \ZZ)^{N_r} \text{\; \; and \; \;} 
(G, q) \simeq (H_1, \mu_1) \bot \dotsb \bot (H_m, \mu_m)
\end{equation*}
where each  $H_i \simeq \ZZ/2^{r_i} \ZZ$ or $H_i \simeq (\ZZ/2^{r_i} \ZZ)^2$ 
and $\mu_i$ is an irreducible non-degenerate quadratic form on $H_i$.
\par
Suppose $k$ is odd. By lemmas \ref{l-gauss-sums-of-irreducible-forms}(b) and \ref{l-theta-of-scaled-form},
if If $(H_i, \mu_i) \cong (\ZZ/2^{r_i}\ZZ, \alpha x^2 / 2^{r_i+1})$, then
\begin{equation*}
\xi_k(H_i, \mu_i) =(-1)^{kr_i(\alpha^2-1)/8} \be(\alpha/8)^k (-1)^{r_i(k^2\alpha^2 - 1)/8}\be(-k\alpha/8).
\end{equation*}
Using equation \eqref{eqn:alpha-beta-16} this simplifies to
\begin{equation*}
\xi_k(H_i, \mu_i) = (-1)^{r_i(k^2-1)/8}.
\end{equation*}
By lemma \ref{l-gauss-sums-of-irreducible-forms} if 
$(H_i, \mu_i) \cong ((\ZZ/2^{r_i}\ZZ)^2, (\alpha x_1^2 + x_1x_2 + \alpha x_2^2)/2^{r_i})$ 
with $\alpha \in \lbrace  0, 1 \rbrace$, then
\begin{equation*}
\xi_k(H_i, \mu_i)  = (-1)^{\alpha^2 r_i k} (-1)^{(-k \alpha)^2 r_i} 
= (-1)^{ \alpha r_i k + \alpha r_i k^2 } = 1.
\end{equation*}
We summarize both cases with the equation
\begin{equation*}
\xi_k(H_i, \mu_i)  = (-1)^{\op{rk}(H_i) r_i(k^2-1)/8}.
\end{equation*}
Summing over all $i$ such that $r_i = r$ yields $\sum_i \op{rk}(H_i) r_i = \sum_r r N_r$. So
\begin{equation*}
\xi_k(G, q)  = (-1)^{ \sum_r r N_r(k^2-1)/8}.
\end{equation*}
The expression for $\xi_k(G,q)$ does not depend on $q$, so we get
$\xi_k(G,q) = \xi_k(G,q')$ for $k$ odd.
\par
Now suppose $k = 2^n \beta$ with $\beta$ odd and $ n > \max\{2, v_2( \exp(G)) \}$.
Then $\op{max} \lbrace r - n , 0 \rbrace = 0$ for all $r$ such that 
$N_r > 0$.  Since $n > v_2(\exp(G))$, the quadratic forms $2^{n-1}b(x,x)$ and
$2^{n-1}b'(x,x)$ are identically equal to zero, so lemma \ref{l-sigma-and-Theta} implies
that $\varsigma_n(b)  = \varsigma_n(b') = 1$.
From lemma \ref{l-two-power-ind}, we get
\begin{equation*}
\xi_{2^n \beta} (G, q)
= \abs{G[2^n]}^{1/2}\varsigma_n(b)^{(2^n-1)\beta}
= \abs{G}^{1/2}.
\end{equation*}
Thus $\xi_{2^n \beta}(G,q)$ does not depend on $q$ and we get  $\xi_{2^n \beta}(G, q) = \xi_{2^n \beta}(G, q')$.
 \end{proof}
Now we are ready to proof theorem  \ref{t-indicators-determine-b}.
\begin{proof}[proof of theorem  \ref{t-indicators-determine-b}]
Write $\cC_1 = \op{TY}(G_1,b_1, \tau_1)$ and $\cC_2 = \op{TY}(G_2, b_2, \tau_2)$.
Let $m_1 = m_{\cC_1}$ and $m_2 = m_{\cC_2}$.
One knows that
 $\op{pdim}(x) = 1$ for $x \in G_j$ and $\op{pdim}(m_j) = \sqrt{\abs{G_j}}$.
 So the hypothesis in the theorem yields
 \begin{equation}
 (\sqrt{\abs{G_1} } - 1)\nu_{k}(m_1) = (\sqrt{\abs{G_2}}  - 1)\nu_{k}(m_2) \text{\; for all \;} k \geq 1.
 \label{eq-gm12}
 \end{equation}
 Lemma \ref{l-nu-and-Theta} implies that if $k$ is a multiple of $8 \abs{G_1} \abs{G_2}$, then
  $\nu_{k}(m_j) = \sqrt{\abs{G_j}}$ for $j = 1,2$.
It follows that  $(\sqrt{\abs{G_1}} - 1)\sqrt{\abs{G_1}}  = (\sqrt{\abs{G_2}} - 1)\sqrt{\abs{G_2}}$ and hence
$\abs{G_1} = \abs{G_2}$. 
\par
First consider the trivial case: $\abs{G_1} = \abs{G_2} = 1$. Then the bilinear forms $b_1$ and $b_2$ are
trivial. So there are only two such $\op{TY}$-categories which are only distinguished by
the value of $\tau \in \lbrace \pm 1 \rbrace$.
We know $\sum_{x \in G_j} \nu_k(x) = \abs{ G_j[k]}$ and $\op{sign}(\tau_j) = \nu_2(m_{\cC_j})$ 
(see theorem 3.2 of \cite{S:FSI} and the remark following the proof of theorem 3.4 of \cite{S:FSI}. Or else,
see lemma \ref{l-nu-and-Theta}).
It follows that $1 + \op{sign}(\tau_1) = \sum_{V \in \op{Irr}{\cC_1}} \nu_2(V) = 
\sum_{V \in \op{Irr}{\cC_2}} \nu_2(V)  = 1 +  \op{sign}(\tau_2)$. So $\op{sign}(\tau_1) = \op{sign}(\tau_2)$.
So the theorem holds in the the trivial case.
\par
We may now assume that $\abs{G_1}  = \abs{G_2} > 1$.
Equation \eqref{eq-gm12} implies $\nu_k(m_1) = \nu_k(m_2)$ and hence 
$\sum_{x \in G_1} \nu_k(x) = \sum_{x \in G_2} \nu_k(x)$ for all $k \geq 1$.
It follows that $\abs{G_1[k]} = \abs{G_2[k]}$  for each $k \geq 1$.
This forces $G_1 \simeq G_2$, and so we may assume without loss of generality that $G_1 = G_2 = G$.
By \cite{S:FSI}, $\op{sign}(\tau_j) = \nu_2(m_{\cC_j})$, and so it follows that $\tau_1 = \tau_2$.  Assume that $b_1$ and $b_2$ are non-isomorphic. 
\par
Write $G = G_e \oplus G_o$ where $G_e$ is the $2$-Sylow subgroup of $G$ and $G_o = \oplus_{p \neq 2} G_{(p)}$
is the ``odd part".
Then $(G, b_j) = (G_o, b_j^o) \bot (G_e, b_j^e)$.
Choose quadratic forms $q_j^o$ and $q_j^e$ such that $b_j^o = \partial q_j^o$
and $b_j^e = \partial q_j^e$. Then $q_j = q_j^o \bot q_j^e$ is a quadratic form such that $\partial q_j = b_j$.
By lemma \ref{l-nu-and-Theta}, it is enough to show that 
$\xi_k(G, q_1) \neq \xi_k(G, q_2)$ for some $k$.
Since $\xi_k$ is multiplicative for $j \in \{1,2\}$ we have
\begin{equation*}
\xi_k( G, q_j) = \xi_k( G_o , q_j^o) \xi_k(G_e, q_j^e).
\end{equation*}
We split the argument in two cases. 
\newline
{\bf Case 1: } Suppose  $b_1^o \ncong b_2^o$.  Then lemma \ref{l-dop} implies that
there is an integer $k > 1$ which is either odd or $v_2(k) > \max\lbrace 2, v_2( \exp(G_e)) \rbrace$ such that
$\xi_k( G_o , q_1^o) \neq \xi_k( G_o , q_2^o)$ and lemma \ref{l-even-parts-agree} implies that
$\xi_k( G_e , q_1^e) =  \xi_k( G_e , q_2^e)$. 
So $\nu_{2k}(m_1) \neq \nu_{2k}(m_2)$ if $b_1^o \ncong b_2^o$.
\par
{\bf Case 2: } Suppose $b_1^o \cong b_2^o$. 
 In this case
we must have $b_1^e \ncong b_2^e$. From theorem \ref{t-KK-quote} (quoted from \cite{KK:LP}),
we know that exists some $n \geq 1$ such that $\sigma_n(b_1^e) \neq \sigma_n(b_2^e)$, which implies $\varsigma_n(b_1^e) \neq \varsigma_n(b_2^e)$.
  Now lemma \ref{l-two-power-ind} implies that 
\begin{equation*}
\xi_{2^n}( G_e, q_j^e) = 
 (-1)^{\Gamma_{G_e, 1, n}}\abs{G_e[2^n]}^{1/2}  \varsigma_n(b_j^e)^{2^n-1}
\end{equation*}
where $\Gamma_{G_e,1,n}$ is an integer dependent on $G_e$ and $n$ but is independent of $q_j^e$.
It follows that $\xi_{2^n}( G_e, q_1^e) \neq \xi_{2^n} (G_e, q_2^e)$.
On the other hand, since 
$(G_o, b_1^o) \cong (G_o, b_2^o)$, 
we have $\xi_{2^n}( G_o, q_1^o) = \xi_{2^n}(G_o, q_2^o)$.
So $\nu_{2^{n+1}}( m_1) \neq \nu_{2^{n+1}}( m_2)$.
 \end{proof}
%
%
%
%
\section{ Tambara-Yamagami categories associated to groups with an odd factor are determined by the state sum invariants}
\label{section-statesum}
%
%
Let $G$ be a finite abelian group, $\chi$ be a non-degenerate symmetric bicharacter on $G$ and $\tau$ be a square
root of $\abs{G}^{-1}$. Let $\cC = \op{TY}(G, \chi, \tau)$ denote the associated Tambara-Yamagami category.
If $M$ is a closed compact 3-manifold, we denote by $\abs{M}_{\cC}$ the 
state-sum invariant of $M$ defined using the category $\cC$, as in \cite{TV:TY}. 
Let $L_{m,n}$ denote the lens spaces.
\begin{lemma}
For all $k \geq 1$, one has $\abs{L_{k,1}}_{\cC} = (\abs{G[k]} + \abs{G}^{1/2} \nu_{k}(m_{\cC} ) )/(2\abs{G})$. 
\label{l-ssi-ind}
\end{lemma}
This lemma follows directly from theorem 0.3 of \cite{TV:TY} as well as lemma \ref{l-nu-and-Theta}.  Theorem 0.3 of \cite{TV:TY} expresses $\abs{L_{2k,1}}_{\cC}$ in terms of a quantity $\zeta_k(\chi)$ which is essentially the right hand side of the equation in lemma \ref{l-nu-and-Theta}.  

%
%
\begin{corollary}
For all $k \geq 1$, one has $\abs{L_{k,1}}_{\cC} = (\op{pdim}(\cC))^{-1} \sum_{V \in \op{Irr}(\cC)} \nu_k(V) \op{pdim}(V)$.
\end{corollary}
The corollary follows from theorem 3.2 of \cite{S:FSI}, which implies $\sum_{x \in G} \nu_k(x) = \abs{G[k]}$.
\begin{theorem}
Let $\cC = \op{TY}(G, \chi, \tau)$ and $\cC'  = \op{TY}(G', \chi', \tau')$ be any two $\op{TY}$-categories. Suppose $\abs{G}$
is not a power of $2$.
If $\abs{L_{k,1}}_{\cC} = \abs{L_{k,1}}_{\cC'}$ for all $k \geq 1$, then $\cC \simeq \cC'$. 
\label{t-Lens}
\end{theorem}
\begin{proof}
 Let $G_{e}$ (resp. $G'_{e}$) be the $2$-Sylow subgroups of $G$ (resp. $G'$). Let 
$G_o$ (resp. $G'_o$) be the sum of the $p$-Sylow subgroups for all odd $p$.
From theorem 0.1 of \cite{TV:TY} we already know that $\abs{G} = \abs{G'}$ and that the $p$-Sylow subgroups of $G$ and $G'$
are isomorphic for all odd $p$. It follows that $\abs{G_e} = \abs{G'_e}$.
We claim that $G_e \simeq G'_e$ as well.
The claim implies $G \simeq G'$ and then lemma \ref{l-ssi-ind} tells us $\nu_k(m_{\cC}) = \nu_k(m_{\cC'})$ for all $k$,
which forces $\chi \simeq \chi'$ by theorem \ref{t-indicators-determine-b}.
Thus to complete the proof we need to show $G_e \simeq G'_e$.
For this it suffices to show that $ \abs{ G[2^n]} =  \abs{ G'[2^n]}$ for all $n \geq 0$.
Suppose this is false. Since
$\abs{G[2^0]} = \abs{G'[2^0]} = 1$,  we may pick
 the smallest $n \geq 0$ such that $ \abs{ G[2^{n+1}]} > \abs{ G'[2^{n+1}]}$ (without loss of generality) and
$ \abs{ G[2^m]} =  \abs{ G'[2^m]}$ for all $m \leq n$. 
\par
Let $a = \abs{G_o} = \abs{G'_o}$.
Let $n \geq 0$.  Then $G[2^n a] = G_o \oplus G[2^n]$. By lemma \ref{l-abs-nu}, we
can write  $\nu_{2^{n+1} a}(m_{\cC} ) = \abs{ G[2^n a] }^{1/2} \psi_n$, where $\psi_n \in \mu_8 \cup \lbrace 0 \rbrace$.
Define $\psi'_n$ similarly for $\cC'$. 
We have
\begin{equation*}
2 \abs{G} \abs{L_{2^{n+1} a, 1} }_{\cC}  
= \abs{G[2^{n+1} a]} + \abs{G}^{1/2} \nu_{2^{n+1} a}(m_{\cC} )
= \abs{G_o} ( \abs{G[2^{n+1}]} + \abs{G_e}^{1/2} \abs{ G[2^n]}^{1/2} \psi_n).
\end{equation*}
So $\abs{L_{2^{n+1} a, 1} }_{\cC}  = \abs{L_{2^{n+1} a, 1} }_{\cC'}$ implies
\begin{equation*}
\abs{G[2^{n+1}]} + \abs{G_e}^{1/2} \abs{ G[2^n]}^{1/2} \psi_n 
= \abs{G'[2^{n+1}]} + \abs{G'_e}^{1/2} \abs{ G'[2^n]}^{1/2} \psi'_n. 
\end{equation*}
If $\psi_n = \psi'_n = 0$, then the above equation would imply 
$\abs{G[2^{n+1}]} = \abs{G'[2^{n+1}]}$. So $\psi_n \neq 0$ or $\psi'_n \neq 0$.  
Rearranging the above equation and remembering that 
$\abs{G_e} = \abs{G_e'}$,  we get
\begin{equation}
\abs{G[2^{n+1}]} - \abs{G'[2^{n+1}]} = 
 \abs{G_e}^{1/2} \abs{ G[2^n]}^{1/2} (\psi'_n  - \psi_n). 
 \label{e-eq-in-cyc}
\end{equation}
Each side of equation \eqref{e-eq-in-cyc} belong to $\ZZ[e^{2 \pi i /8}]$. Consider the absolute norm of each side. 
If $\psi \in \mu_8 \cup \lbrace 0 \rbrace$, one verifies that the absolute norm of $(\psi - 1)$ is a power of $2$ or zero.
For example if $\psi$ is a primitive $8$-th root of unity, then 
$
N^{\QQ[\psi]}_{\QQ}(\psi - 1) =  \prod_{j = 0}^3 (\be(\tfrac{2 j + 1}{8}) - 1)  = 2.
$
If $\psi_n \neq 0$ (resp. $\psi'_n \neq 0$)  then writing $(\psi'_n  - \psi_n) = \psi_n( \psi'_n/ \psi_n - 1)$  (resp.
$(\psi'_n  - \psi_n) = \psi'_n(1 -  \psi_n/ \psi'_n)$) we find that the norm of $( \psi'_n - \psi_n)$
is a power of $2$ or $0$.
So the norm of the right hand side of equation \eqref{e-eq-in-cyc} is also
a power of $2$. However, note that the left hand side is already an integer, so it must also be a power of $2$.
The only way this is possible is if $\abs{G[2^{n+1}] }= 2 \abs{G'[2^{n+1}]}$.
Write $\nu_{2^{n+1} } (m_{\cC} ) = \abs{G[2^n]}^{1/2} \lambda_n$ and
$\nu_{2^{n+1} } (m_{\cC'} ) = \abs{G'[2^n]}^{1/2} \lambda'_n$ for some 
$\lambda_n, \lambda'_n \in \mu_8 \cup \lbrace 0 \rbrace$.
Now the equality $\abs{L_{2^{n+1} , 1}}_{\cC} =  \abs{L_{2^{n+1} , 1}}_{\cC'}$ yields
\begin{equation*}
\abs{G'[2^{n+1}]} = 
\abs{ G[2^{n+1}] } - \abs{G'[2^{n+1}] } = \abs{G}^{1/2} \abs{G[2^n]}^{1/2} (\lambda'_n - \lambda_n).
\end{equation*}
Now the left hand side is a power of $2$, so the norm of the right hand side must also be a power of $2$.
Since $N( \lambda_n' - \lambda_n)$ is a power of $2$, it follows that $\abs{G}$ is also a power of $2$
which is against our assumption. It follows that $(G, \chi) \simeq (G', \chi')$. Now since $\nu_2(m_{\cC} ) = \op{sgn}(\tau)$,
the equality $\abs{L_{2,1}}_{\cC} = \abs{L_{2,1}}_{\cC'}$ implies $\tau = \tau'$.
 \end{proof}
\par
{\bf Example.}  We exhibit two Tambara Yamagami categories that have the same state-sum invariant for all lens spaces $L_{k,1}$.  
Recall that  $A_{2^n}$ denotes the metric group $((\ZZ/2^n \ZZ), x^2/2^{n+1})$.
For $k \in \ZZ$, we shall denote the pre-metric group $((\ZZ/2^n \ZZ), k x^2/2^{n+1})$ by 
 $(k \cdot A_{2^n})$.
 Let $(G_1, b_1) = (A_{2})^4  \perp A_{4} $ and $(G_2, b_2) = (A_{2})^2  \perp (A_{4})^2$.
Let $\cC_1 = \op{TY}(G_1, b_1, -\tfrac{1}{8})$ and $\cC_2 = \op{TY}(G_2, b_2, \tfrac{1}{8})$.  Then
we claim that 
$\abs{L_{n,1}}_{\cC_1} = \abs{L_{n,1}}_{\cC_2}$
for all positive integers $n$.
\begin{proof}[proof of claim]
Let $q_i$ be a quadratic form such that $\partial q_i = b_i$ for $i \in \lbrace 1,2 \rbrace$.  We will break the proof into cases
according to possible $2$-valuations of $n$.  
The trivial case is that $\abs{L_{n,1}}_{\cC_1} = \frac{1}{128} = \abs{L_{n, 1}}_{\cC_2}$ if $n$ is odd. 
By lemma \ref{l-ssi-ind} and lemma \ref{l-nu-and-Theta}, to prove 
$\abs{L_{2k,1}}_{\cC_1} = \abs{L_{2k,1}}_{\cC_2}$ it is enough to show that
\begin{equation*}
\abs{G_1[2k]} + (-1)^k 8 \xi_k(G, q_1) 
= \abs{G_2[2k]} + 8 \xi_k(G,q_2)
\end{equation*}
Since $\xi_k$ is multiplicative,
\begin{equation*}
\xi_k(G, q_1)
= \xi_k(A_2 )^4 \xi_k( A_4)  \text{\; \; and \; \;} 
\xi_k(G, q_2)
= \xi_k(A_2)^2 \xi_k(A_4)^2.
\end{equation*}
\par
From lemma \ref{l-gauss-sums-of-irreducible-forms}, we have 
$\xi_k(A_{2^r}) = \Theta(A_{2^r} )^k \Theta( -k \cdot A_{2^r}) = e(k/8) \Theta(-k \cdot A_{2^r} )$.
The values of $\Theta(-k \cdot A_{2^r} )$ were computed in lemma \ref{l-theta-of-scaled-form}. 
This lets us compute the invariants. We shall make three cases:
\newline
{\bf Case 1: } Suppose $k$ is odd. Then we have
$\Theta(-k \cdot A_{2} ) = (-1)^{(k^2 - 1)/8} \be(-k/8)$, so $\xi_k(A_2) = (-1)^{(k^2 - 1)/8}$. 
We have $\Theta(-k \cdot A_{4} ) = (-1)^{2(k^2 - 1)/8} \be(-k/8) = \be(-k/8)$, so $\xi_k(A_4) = 1$.
It follows that $\xi_k(G, q_1) = 1  = \xi_k(G, q_2)$.  
Since  $\abs{G_1[2k]} = 32$ and $\abs{G_2[2k]} = 16$, we get 
$\abs{L_{2k,1}}_{\cC_1} = \abs{L_{2k,1}}_{\cC_2}$ in this case.
\par
{\bf Case 2: } Suppose $v_2(k) = 1$ or $2$. Then $\Theta(-k \cdot  A_2) = 0$ or $\Theta(-k \cdot A_4) = 0$, so $\xi_k( A_2) = 0$ or $\xi_k(A_4) = 0$.
Since both $(G_1, b_1)$ and $(G_2, b_2)$ have components of type $A_2$ and $A_4$ and since $\xi_k$ is multiplicative, it follows that
$\xi_k(G, q_1) = \xi_k(G, q_2) = 0$.
Since $\abs{G_i[2k]} = 64$, we get 
$\abs{L_{2k,1}}_{\cC_1} = \abs{L_{2k,1}}_{\cC_2}$ in this case.
\par
{\bf Case 3: } Finally suppose $v_2(k) \geq 3$. Let $r = 1$ or $r = 2$.
Then $\Theta( A_{2^r})^k = e(k/8) = 1$. The quadratic form $-k \cdot A_{2^r}$ 
is identically equal to $1$, so $\xi_k(A_{2^r}) = \Theta(-k \cdot A_{2^r} ) = 2^{r/2}$. It follows that 
$\xi_k(G, q_j) = \abs{G}^{1/2} = 8$ for $j = 1, 2$.
Since $\abs{G_i[2k]} = 64$ and $(-1)^k = 1$, we get $\abs{L_{2k,1}}_{\cC_1} = \abs{L_{2k,1}}_{\cC_2}$ in this case too.
\end{proof}
%
%

%
%
\appendix
%
%
%
%
\section{Diagonalization of bilinear and quadratic forms}
\label{app-A}
%
%
In this appendix, we discuss the problem of decomposing quadratic and bilinear forms on finite abelian groups
into irreducible components. 
\newline
\par
{\bf Some notation: }
If $R$ is an abelian group, we let $M_n(R)$
be the set of all $n \times n$ matrices with entries in $R$. 
If $R$ be a commutative ring and $S$ is an $R$-module, then
$S^n$ is a (left) $M_n(R)$-module and $M_n(S)$ is a $M_n(R)$-bimodule. 
The action of $M_n(R)$ on $S^n$ is obtained by writing elements
of $S^n$ as column vectors and multiplying by the matrix on the left. 
The two actions of $M_n(R)$ on $M_n(S)$ are by left and right multiplication.
\newline
\par
Recall from section \ref{section-bq} that if $x$ is an element in a $p$-group of finite order, then we write $v_p(x) = - \log_p( \op{order}(x))$ and $v_p(0) = \infty$.
The lemma below is elementary. We leave the proof as an easy exercise.
\begin{lemma}
Let $p$ be a prime. Let $G$ be an abelian $p$-group.
\par
(a) Let $x \in G$ and $r \in \ZZ$. Then $r x = 0$ if and only if $v_p(r) + v_p(x)  \geq 0$.
\par
(b) If $x \in G$ and $r \in \ZZ$ such that $r x \neq 0$, then $ v_p(r) + v_p(x)= v_p(rx)$.
\par
(c) Let $x_1, x_2 \in G$. Then  $v_p( x_1 + x_2 ) \geq \op{min} \lbrace v_p(x_1) , v_p(x_2) \rbrace$
and equality holds if $\langle x_1 \rangle \cap \langle x_2 \rangle = 0$ or $v_p(x_1) \neq v_p(x_2)$.
(here and later, $\langle x \rangle$ denotes the cyclic subgroup generated by $x$)
\par
(d) Let $b$ be a symmetric bilinear form on a finite abelian $p$-group $G$.
If $g \in G$, then $v_p(g) \leq v_p( b(g, h))$ for all $h \in G$. Further, if $b$ is non-degenerate, then 
$v_p(g) = \min \lbrace v_p( b(g, h)) \colon h \in G \rbrace$.
\label{l-pval}
\end{lemma}
Decomposing symmetric bilinear forms into irreducible components is almost equivalent to diagonalizing matrices
by row and column operations. We introduce these operations next.
\newline
\par
{\bf Definition.} Let $E_{ij}$ be the $n \times n$ matrix whose $(i,j)$-th entry is $1$ and all other entries are $0$.
Let $I_n$ denote the $n \times n$ identity matrix.
Let $R$ be a commutative ring. Let $A$ be a $n \times n$ matrix with entries in some $R$-module $M$.
The operations $\op{Flip}_{i j}(A)$, $\op{Add}^{r,j}_{i}(A)$, and $\op{Scale}^r_i(A)$ defined below are called 
{\it row-column operations} on $A$:
\begin{itemize}
\item Let $\op{Flip}_{ij}(A) = S^{tr} A S$ where $S = I_n - E_{i i} - E_{j j} + E_{i j} + E_{j i}$.
This operation interchanges the $i$-th and $j$-th rows of $A$ and then interchanges the $i$-th and $j$-th columns of $A$.
\item Let $\op{Add}^{r,j}_i(A) =  S^{tr} A S$, where $S = I_n + r E_{ji}$ for some $r \in R$ and $i \neq j$.
This operation  adds $r$ times the $j$-th row of $A$ to the $i$-th row of $A$ and then adds $r$ times the
$j$-th column of $A$ to the $i$-th column of $A$.
\item Let $\op{Scale}^r_i(A) = S^{tr} A S$ where $S = I_n + (r - 1) E_{i i}$ for some $r \in R$. This operation multiplies
the $i$-th row of $A$ by $r$ and then multiplies the $i$-th column by $r$.
\end{itemize}
Let $(G, b)$ be a discriminant form and $(e_1, \dotsb, e_n) \in G^n$. For each $i \neq j$,
the operation $\op{Flip}_{i j}$ converts $\op{Gram}_b( e_1,...,e_n)$ to $\op{Gram}_b(f_1,...,f_n)$ where 
$f_j = e_i$, $f_i = e_j$  and $f_k = e_k$ for $k \notin \lbrace i, j \rbrace$. The operation $\op{Add}^{r,j}_i$ 
converts $\op{Gram}_b( e_1,...,e_n)$ to $\op{Gram}_b(f_1,...,f_n)$ where 
$f_i = e_i + r e_j$ and $f_k = e_k$ for $k \neq i$.
The operation $\op{Scale}^r_i$
converts $\op{Gram}_b( e_1,...,e_n)$ to $\op{Gram}_b(f_1,...,f_n)$ where 
$f_i = r e_i $ and $f_k = e_k$ for $k \neq i$.
We shall say that a row-column operation on $\op{Gram}_b( e_1,...,e_n)$ 
is {\it valid} if $G = \oplus_k \langle e_k  \rangle$ implies $G = \oplus_k \langle f_k \rangle$.
Clearly, $\op{Flip}_{i j}$ is always valid. The operation $\op{Scale}^r_i$ is valid if $r$ is relatively
prime to the exponent of $G$. Lemma \ref{l-abgg} lets us decide when 
$\op{Add}^{r,i}_j$ is valid.
\begin{lemma}
Let $G$ be a finite abelian group and $e_1,\dotsb, e_n \in G$ such that $G = \oplus_k \langle e_k \rangle$.
Let $f_1, \dotsb, f_n \in G$ such that $\op{ord}(f_k) = \op{ord}(e_k)$ for all $k$ and $f_1, \dotsb, f_n$
generate $G$. Then there exists $\phi \in \op{Aut}(G)$ such that $\phi(e_k) = f_k$. In particular,
$G = \oplus_k \langle f_k \rangle$.
\label{l-abgg}
\end{lemma}
\begin{proof}
Let $n_k  = \op{ord}(e_k) = \op{ord}(f_k)$. Since $\langle e_k \rangle$ is a cyclic group of order $n_k$ and
$f_k$ is an element of order $n_k$ in $G$, there exist  a homomorphism $\phi_k : \langle e_k \rangle \to G$
given by $\phi_k( e_k) = f_k$. By universal property of
direct sum, there exists a homomorphism $\phi: G \to G$ such that $\phi(e_k) = f_k$ for all $k$. Since the $f_k$'s generate
$G$, the map $\phi$ is onto. Since $G$ is finite group, $\phi$ must be injective as well. 
\end{proof}
let $A \in M_n( \QQ_{(p)}/\ZZ)$. 
The proofs of the next two lemmas \ref{l-diagonalization} and \ref{l-diagonalization2} are based on the algorithm to
reduce $A$ to a diagonal matrix 
(or a block diagonal matrix with blocks of size at most two when $p =2$) by conjugation or equivalently using the
 elementary row-column operations introduced above. This paves the way to prove theorem \ref{t-wall} of \cite{W:QF}.
 Let $\op{diag}(a_1, \dotsb, a_n)$ denote the diagonal $n$ by $n$ matrix with diagonal entries $a_1, \dotsb, a_n$.
\begin{lemma} 
Let $p$ be an odd prime. Let $u_p$ be a quadratic non-residue modulo $p$.
Let $A \neq 0$ be a symmetric matrix in $M_n(  \QQ_{(p)}/\ZZ)$. Let $r_1$ be the smallest number such that $p^{r_1} A = 0$.
\par
(a) Then there exists a matrix $S \in \op{GL}_n(\ZZ)$ such that $S \bmod p \in \op{GL}_n( \ZZ/ p \ZZ)$ and 
\begin{equation*}
S^{tr} A S = \op{diag}( p^{-r_1} \epsilon_1, \dotsb, p^{-r_n} \epsilon_n), \text{\; with\;}
r_1\geq r_2 \geq \dotsb \geq r_n \geq 0, \; \epsilon_j \in \lbrace 1, u_p, 0 \rbrace, \; \epsilon_1 \neq 0.
\end{equation*}
\par
(b) Let $(G,b)$ be a non-degenerate discriminant form where $G$ is a $p$-group.
Let $G = \oplus_{j = 1}^n \langle e_j \rangle$.
 Then there exists $f_1, \dotsb, f_n \in G$ such that
$G = \oplus_{j = 1}^n \langle f_j \rangle$ and 
$\op{Gram}_b( f_1, \dotsb, f_n) = \op{diag}( p^{-r_1} \epsilon_1, \dotsb, p^{-r_n} \epsilon_n)$
with  $r_1\geq r_2 \geq \dotsb \geq r_n > 0$, 
$\epsilon_j \in \lbrace 1, u_p \rbrace$. 
\label{l-diagonalization}
\end{lemma}
\begin{proof} (a) One proceeds by finding a pivot with the smallest 
$p$-valuation and then using this pivot to sweep out the rows and columns. 
Let $A = ( \!( a_{i j} ) \!)  \in M_n( \QQ_{(p)}/\ZZ)$ be a symmetric non-zero matrix.
Let $r_1 > 0$ be the smallest integer such that $p^{r_1} A = 0$.
 By induction on $n$, it suffices to show that there is a sequence of row-column operations that converts $A$
to a matrix of the form $\bigl( \begin{smallmatrix}  d_{1} & 0 \\ 0 & A'  \end{smallmatrix} \bigr) $
where $d_1  = p^{-r_1}$ or $d_1 = u_p p^{-r_1}$ 
and $A'   \in M_{n-1}( \QQ_{(p)}/\ZZ)$ is a symmetric  matrix  such that $p^{r_1} A' = 0$.
\par
{\bf Finding a pivot:} {\it We claim that after changing $A$ by row-column operations, we may assume
that $a_{1 1} = p^{-r_1}$ or $a_{ 1 1}  = u_p p^{-r_1}$.}
\newline
{\it proof of claim: } If there is a diagonal entry $a_{i i}$ such that $v_p(a_{i i} ) = -r_1$, then apply $\op{Flip}_{1 i}$ to $A$
to get $v_p(a_{1 1} ) = -r_1$. 
Otherwise, there exists $i \neq j$ such that $v_p(a_{i j}) = -r_1$ and $v_{p} (a_{ii}) > -r_1 , v_p(a_{j j}) > - r_1$.
In this case, apply $\op{Add}_{i}^{1 , j}$ to $A$.
This changes the $(i,i)$-th entry of the matrix from $a_{i i}$ to $(a_{ i i}  + 2 a_{i j} + a_{j j})$ whose $p$-valuation is $-r_1$
\footnote{this is the step in the argument that fails for $p=2$.}.
Now, we apply $\op{Flip}_{1 i}$. Either way, we get $v_p(a_{1 1}) = -r_1$.
Using the operation $\op{Scale}^r_i$ we can change $a_{1 1}$ to $r^2 a_{1 1}$. By choosing $r$ appropriately, we
can make $a_{1 1} = p^{-r_1}$ or $a_{ 1 1}  = u_p p^{-r_1}$.
\par
{\bf Sweeping out: } Now $a_{1 1} = \epsilon_1  p^{-r_1}$ with $\epsilon_1 = 1$ or $u_p$.
Since $\epsilon_1$ is relatively prime to $p$, we can pick $\epsilon' \in \ZZ$ such that
$\epsilon' \epsilon_1 \equiv 1 \bmod p^{r_1}$. 
We can represent $a_{1i}$ in the form $\beta_i  p^{-r_1}$ with $\beta_i \in \ZZ$.
We add $(- \beta_i \epsilon')$ times the first row to the $i$-th row and then 
add $(- \beta_i \epsilon')$ times the first column to the $i$-th column to 
make $a_{1i} = 0$ and $a_{i1} = 0$. By performing this operation for $ i = 2, 3, \dotsb, n$
converts $A$ to a  matrix of the form
$\bigl( \begin{smallmatrix}  \epsilon_1 p^{-r_1} & 0 \\ 0 & A'  \end{smallmatrix} \bigr) $.
Finally note that the entries of $A'$ are $\ZZ$-linear combinations of entries of $A$, so $p^{r_1} A = 0$
implies $p^{r_1} A' = 0$. The row-column operations above correspond to conjugating $A$ by certain matrices
which are always invertible modulo $p$. Now part (a) follows by induction.
\par
(b) Assume the setup of part (b).  Let $A = \op{Gram}(e_1, \dotsb , e_n)$.
Part (a) shows that the matrix $A$ can be diagonalized by a sequence 
of row-column operations. Performing a row-column operation on $\op{Gram}_b(e_1,...,e_n)$
converts it to $\op{Gram}_b(f_1,...,f_n)$ where $f_j$'s are given in definition preceding lemma \ref{l-abgg}.
We  need to verify that all the row-column operation used in the proof of part (a) are valid
(see the definition preceding lemma \ref{l-abgg}).
While finding the pivot,  we may perform $\op{Add}^{1,j}_i$ to a matrix $\op{Gram}(e_1,...,e_n)$ 
if a non-diagonal entry of the matrix, say $a_{i j}$, has the highest power of $p$
 in the denominator. Since $a_{i j} = a_{j i}$, lemma \ref{l-pval}(d) implies that 
$\op{order}(e_i)  =  \op{order}(e_j)$. Since $\langle e_i \rangle \cap \langle e_j \rangle = 0$,
lemma \ref{l-pval} implies that $\op{ord}(e_i + e_j) = \op{ord}(e_i)$.
Now lemma \ref{l-abgg} implies that $\op{Add}^{1,j}_j$ is valid.
\par
While sweeping out, we perform the row-column operation $\op{Add}^{-\beta_i \epsilon', 1}_i$ where 
$a_{1 i} = \beta_i p^{-r_1}$. This operation changes $\op{Gram}(e_1,...,e_n)$ to $\op{Gram}(f_1,...,f_n)$ where 
$f_i = e_i -\beta_i \epsilon' e_1$ and $f_k = e_k$ for $k \neq i$. Assume $G = \oplus_k \langle e_k \rangle$.
 Since the discriminant form on $G$ is non-degenerate, we have
$v_p(e_1) = -r_1$ and hence $v_p( -\beta_i \epsilon' e_1) = v_p(\beta_i) - r_1$.
Also, $v_p( e_i) \leq v_p(a_{1 i}) =  v_p(\beta_i) - r_1$. 
Since $\langle e_i \rangle \cap \langle  -\beta_i \epsilon' e_1 \rangle = \lbrace 0 \rbrace$, we have
$v_p( f_i ) = \min \lbrace v_p(e_i), v_p( -\beta_i \epsilon' e_1) \rbrace = v_p(e_i)$. Lemma
\ref{l-abgg} implies that the row-column operations performed while sweeping out are valid.
\par
It follows that there exists $f_1, \dotsb, f_n \in G$ such that $G = \oplus \langle f_j \rangle$ and 
$\op{Gram}_b( f_1, \dotsb, f_n) = \op{diag}( p^{-r_1} \epsilon_1, \dotsb, p^{-r_n} \epsilon_n)$
with  $r_1\geq r_2 \geq \dotsb \geq r_n \geq 0$, 
$\epsilon_j \in \lbrace 1, u_p, 0\rbrace$. Since $(G,b)$ it non-degenerate, it follows that we must have
$\epsilon_j \neq 0$ and $\op{order}(f_j) = p^{r_j}$ for all $j$.
\end{proof}
The next lemma handles the case of the prime $p = 2$. 
This proof is similar to the proof of lemma \ref{l-diagonalization}, but somewhat more complicated. 
We only elaborate on the modifications needed to the proof of lemma \ref{l-diagonalization}.
\begin{lemma} 
(a) Let $A \neq 0$ be a symmetric matrix in $M_n(  \QQ_{(2)}/\ZZ)$. Let $m$ be the smallest number such that $2^{m} A = 0$.
Then there exists a matrix $S \in \op{GL}_n(\ZZ)$ such that $(S \bmod 2) \in \op{GL}_n( \ZZ/ 2 \ZZ)$ and 
$S^{tr} A S$ is block diagonal with blocks of size $1$ or $2$. Each block is of the form
\begin{equation}
\bigl( 2^{-r}  \delta \bigr), \;\;\text{or}\;\; 2^{-r} \bigl( \begin{smallmatrix} 2 a & b \\ b & 2c \end{smallmatrix} \bigr)
\label{eq-nondegform2}
\end{equation}
where $r$ is some non-negative integer, $a, b, c$ are integers with $b$ odd and $\delta \in \lbrace 0, \pm 1, \pm 5 \rbrace$.
The largest $r$ that occurs is equal to $m$.
\par
(b) Let $(G,b)$ be a non-degenerate discriminant form where $G$ is a $2$-group.
Let $G = \oplus_{j = 1}^n \langle e_j \rangle$. Then there exists $f_1, \dotsb, f_n \in G$ such that
$G = \oplus_{j = 1}^n \langle f_j \rangle$ and 
$\op{Gram}_b( f_1, \dotsb, f_n)$ is a block diagonal matrix with with blocks of size one or two.
Each block is of the form given in \eqref{eq-nondegform2}
where $r$ is some positive integer, $a, b, c$ are integers with $b$ odd and $\delta \in \lbrace \pm 1, \pm 5 \rbrace$.
\label{l-diagonalization2}
\end{lemma}
\begin{proof}
(a) As above, we try to get a diagonal entry of $A$ to have minimum $2$-valuation.
If this succeeds, then we can proceed with the sweep out as before and split off a one-by-one block from
$A$. This procedure fails only in the situation when there exists $i \neq j$ such that
$ \bigl( \begin{smallmatrix} a_{i i} & a_{i j} \\ a_{j i} & a_{j j} \end{smallmatrix} \bigr) = 
2^{-m} \bigl( \begin{smallmatrix} 2 \alpha & \beta \\ \beta &  2 \gamma \end{smallmatrix} \bigr)$ with
$\alpha, \beta, \gamma \in \ZZ$, 
$\beta$ odd and all the diagonal entries of $A$ have $2$ valuation strictly larger than $-m$.
In this case, we can use row-column flips
 to move this $2 \times 2$ sub-matrix to the upper left corner of $A$ so that 
$\bigl( \begin{smallmatrix} a_{1 1} & a_{1 2} \\ a_{2 1} & a_{2 2}  \end{smallmatrix} \bigr)
=2^{-m} \bigl( \begin{smallmatrix} 2 \alpha & \beta \\ \beta &  2 \gamma \end{smallmatrix} \bigr)$ 
and then use this $2 \times 2$ block to sweep out the first two rows and first two columns simultaneously.
\par
This is how it is done: Suppose the first two entries of the $i$-th row are $2^{-m} (u,v)$ for $u, v \in \ZZ$
where $i > 2$. We want to find $r_1, r_2$ such that
\begin{equation*}
(r_1, r_2) 2^{-m} \bigl( \begin{smallmatrix} 2 \alpha & \beta \\ \beta &  2 \gamma \end{smallmatrix} \bigr)
=2^{-m} (u,v) \mod \ZZ.
\end{equation*}
This system can always be solved since the determinant $(4 \alpha \gamma - \beta^2)$ of the
coefficient matrix is odd. Solving the equation yields 
yields
\begin{equation*}
(r_1, r_2) = d ( 2 \gamma u - \beta v, 2 \alpha v - \beta u) 
\end{equation*}
where $d$ is an inverse of $(4 \alpha \gamma - \beta^2)$ modulo $2^m$.
Now we  add to the $i$-th row $-r_1$ times the first row and $-r_2$ times the second row and then perform
the corresponding column operations to the $i$-th column. Verify that after these operations the first two entries of the
$i$-th row and $i$-th column become zero. Part (a) follows.
\par
(b) Let $A = \op{Gram}(e_1, \dotsb, e_n)$.
The sweep out operation described above corresponds to replacing $\op{Gram}(e_1, \dotsb, e_n)$
by $\op{Gram}(f_1,\dotsb, f_n)$ where
$f_i = e_i + r_1 e_1 + r_2 e_2$ and $f_j = e_j$ for all $j \neq i$. The extra work needed in part (b)
is to check that this operation is valid. Note that since $2^m$ is the maximum denominator in $A$, 
$\op{order}(e_1) = \op{order}(e_2) = 2^m$. Suppose $\op{order}(e_i) = 2^k$. Then $u$ and $v$ must be
divisible by $2^{m - k}$ because the entries of the $i$-th row can have denominator at most $2^k$. 
From the formula for $r_1$ and $r_2$ we see that $2^{m - k}$ divides $r_1$ and $r_2$. It follows
that $2^k f_i = 0$. On the other hand, since $\langle e_i \rangle \cap \langle e_1, e_2 \rangle = 0$,
we have $\op{order}(f_i) \geq 2^k$. So $\op{order}(f_i) = \op{order}(e_i)$ and lemma \ref{l-abgg} implies the sweep out operations
using $2 \times 2$ blocks described above are valid.
\end{proof}
For $p$-groups with $p$ odd, Wall's theorem \ref{t-wall}(a) follows from lemma \ref{l-diagonalization}. For $p =2$, we need lemma
\ref{l-diagonalization2} and we also need the lemmas \ref{l-rank-two-irr-q} and \ref{l-rank-two-irr-b},
that describe the irreducible non-degenerate quadratic and bilinear forms on $(\ZZ/2^r \ZZ)^2$.
Proving lemmas \ref{l-rank-two-irr-q} and \ref{l-rank-two-irr-b} depends on solving a system of congruence equations
modulo $2^n$ for all $n$. This can be done by a  standard application of Hensel's lemma.
First we state Hensel's lemma in the form we need.
\begin{lemma}[Hensel's lemma]
Let $p$ be a prime.
Let $f_1, \dotsb, f_m \in \ZZ[x_1, \dotsb, x_n]$ and $f = (f_1, \dotsb, f_m)$.
Let $Df = (\!( \partial f_i/ \partial x_j )\!)$ be the Jacobian of $f$.
Let $t_1 \in \ZZ^n$ such that
$f(t_1) \equiv 0 \bmod p$ and the $m \times n$ matrix 
$(Df(t_1) \mod p)$ has rank $m$ over $\mathbb{F}_p$. 
Then, for all $k \geq 1$, there exists $t_k \in \ZZ^n$ such that $t_{k +1} \equiv t_k \bmod p^k$ 
and $f(t_k) \equiv 0 \bmod p^k$.
\end{lemma}
The proof is omitted.
%
%
\begin{lemma}
(a) Let $s = \bigl( \begin{smallmatrix} s_{1 1} & s_{1 2} \\ s_{2 1} & s_{2 2} \end{smallmatrix} \bigr)$
be a $2 \times 2$ matrix of indeterminates. Let 
\begin{equation*}
(A(s) , B(s), C(s) )= 
(s_{1 1} ^2 + s_{1 1} s_{ 1 2} + s_{1 2} ^2, 
2 s_{1 1} s_{2 1} + s_{1 1} s_{2 2} + s_{2 1} s_{1 2} + 2 s_{1 2} s_{2 2}, 
s_{2 1} ^2 + s_{2 1} s_{2 2} + s_{2 2}^2).
\end{equation*}
Let $A, B, C$ be odd integers. Let $n \geq 1$. Then  the equation
\begin{equation}
(A(s), B(s), C(s)) \equiv (A, B, C) \bmod 2^n 
\label{eq-Fn}
\end{equation}
has a solution $S \in M_2(\ZZ)$ such that $S \equiv I \bmod 2$ (here $I$ denotes the $2 \times 2$ identity matrix).
\par
(b) Let $s = \bigl( \begin{smallmatrix} s_{1 1} & s_{1 2} \\ s_{2 1} & s_{2 2} \end{smallmatrix} \bigr)$
be a $2 \times 2$ matrix of indeterminates. Let 
\begin{equation*}
(A(s), B(s), C(s) ) = (s_{1 1} s_{ 1 2},  s_{1 1} s_{2 2} + s_{2 1}s_{1 2} ,  s_{2 1} s_{2 2}) .
\end{equation*}
Let $A, B, C$ be integers such that $B$ is odd and $AC$ is even. Let $n \geq 1$. Then  the equation
\begin{equation}
(A(s), B(s), C(s)) \equiv (A, B, C) \bmod 2^n 
\label{eq-En}
\end{equation}
has a solution $S \in M_2(\ZZ)$ such that $S \equiv \bigl( \begin{smallmatrix} A & 1 \\ 1 & C \end{smallmatrix} \bigr) \bmod 2$.
\label{l-congruences}
\end{lemma}
\begin{proof}
(a) Apply Hensel's lemma to $f = (f_1, f_2, f_3)$ where
$f_1(s)  = s_{1 1} ^2 + s_{1 1} s_{ 1 2} + s_{1 2} ^2 - A$,
$f_2(s) = 2 s_{1 1} s_{2 1} + s_{1 1} s_{2 2} + s_{2 1} s_{1 2} + 2 s_{1 2} s_{2 2} - B$,
$f_3(s) = s_{2 1} ^2 + s_{2 1} s_{2 2} + s_{2 2}^2 - C$.
Since $A, B, C$ are odd, $ s = \op{I}$ is a solution to $f(s) \equiv 0 \bmod 2$.
One computes 
\begin{equation*}
Df = \begin{pmatrix}
2 s_{1 1} + s_{1 2} &    0        &   s_{1 1} + 2 s_{1 2} & 0 \\
2 s_{2 1} + s_{2 2} &  2s_{1 1} + s_{1 2} &   s_{2 1} + 2 s_{2 2}  & s_{1 1} + 2 s_{1 2} \\
0         & 2 s_{2 1} + s_{2 2} &  0             & s_{ 2 1} + 2 s_{2 2}
\end{pmatrix},
\text{so \;} 
Df(I)  \equiv \begin{pmatrix}
0 &  0  & 1 & 0 \\
1 &  0 &  0 & 1 \\
0 & 1  &  0 & 0
\end{pmatrix} \bmod 2
\end{equation*}
which has rank 3. For part (b), let 
$f_1(s) = s_{1 1} s_{ 1 2} - A$, $f_2(s)= s_{1 1} s_{2 2} + s_{2 1}s_{1 2} - B $, $f_3(s) =  s_{2 1} s_{2 2} - C$.
Since $B$ is odd and $A C$ is even, $s_* = \bigl( \begin{smallmatrix} A & 1 \\ 1 & C \end{smallmatrix} \bigr)$
satisfies $f(s_*) \equiv 0 \bmod 2$.
One computes
\begin{equation*}
Df =  
\begin{pmatrix}
s_{1 2} &    0        &   s_{1 1}  & 0 \\
s_{2 2} &  s_{1 2} &   s_{2 1}   & s_{1 1}  \\
0         & s_{2 2} &  0             & s_{ 2 1} 
\end{pmatrix},
\text{\; so \;} 
Df(s_*)  \equiv 
\begin{pmatrix}
1 & 0  &   A & 0 \\
C &  1 &   1 & A \\
0  & C &   0   & 1
\end{pmatrix} \bmod 2.
\end{equation*}
Since $A$ or $C$ is even, either the second or the third column of the above matrix is equal to $(0, 1, 0)^{tr}$. So the matrix $(Df(s_*) \bmod 2)$ has rank $3$.
\end{proof}
\begin{proof}[proof of lemma \ref{l-rank-two-irr-q}]
(a) Note that $2q(x) = \partial q(x,x) \in 2^{-r} \ZZ/ \ZZ$. So $q(x)$ takes values in $2^{-r -1} \ZZ / \ZZ$, and
\begin{equation*}
q(x_1 ,x_2) = 2^{-r - 1}(\alpha x_1^2 + 2 B x_1 x_2 + \gamma x_2^2)
\end{equation*}
 where $q(1,0) =  2^{-r - 1} \alpha $, $ q(0,1) = 2^{-r - 1} \gamma $ and $\partial q( (1,0), (0,1)) = 2^{-r} B$.
Suppose $\alpha$ is odd. Let $\bar{\alpha}$ be an inverse of $\alpha$ modulo $2^{r +1}$.
Then we can complete squares to write
\begin{equation*}
q(x_1 , x_2) = 2^{-r - 1} ( \alpha ( x_1 + B \bar{\alpha} x_2)^2 + (\gamma - B^2 \bar{\alpha}) x_2^2). 
\end{equation*}
This contradicts the irreducibility of $q$, and thus $\alpha$ has to be even. For the same reason $\gamma$ has to be even. 
So we can write
\begin{equation*}
q(x_1 ,x_2) = 2^{-r} (A x_1^2 +  B x_1 x_2 + C x_2^2).
\end{equation*}
If $A$, $B$, $C$ are all even, then $\partial q$ takes values in $2^{-r + 1} \ZZ/ \ZZ$
and hence cannot be non-degenerate. 
If $B$ is even,  then $A$ or $C$ must be odd,
and we can once again complete squares (as above) and decompose $(G,q)$ into orthogonal direct sum of
two metric groups.  So $B$ must be odd. 
\par
First, suppose $AC$ is odd. Let $F(x_1, x_2) = x_1^2 + x_1 x_2 + x_2^2$. Let 
$s = \bigl( \begin{smallmatrix} s_{1 1} & s_{1 2} \\ s_{2 1} & s_{2 2} \end{smallmatrix} \bigr)$.
Note that 
\begin{equation*}
F((x_1, x_2) s) = A(s) x_1^2 + B(s) x_1 x_2 + C(s) x_2^2
\end{equation*}
where $(A(s), B(s), C(s))$ are the polynomials given in lemma \ref{l-congruences}(a). 
We want to show $q(x_1, x_2) \simeq 2^{-r} F(x_1, x_2)$. This is equivalent to finding a matrix $s \in M_2(\ZZ)$
with odd determinant such that 
\begin{equation*}
F( (x_1, x_2) s )  \equiv (A x_1^2 + B x_1 x_2 + C x_2^2) \bmod 2^r,
\end{equation*}
or equivalently, $(A(s), B(s), C(s)) \equiv (A, B, C) \bmod 2^r$.
The proof follows from lemma \ref{l-congruences}(a), if $AC$ is odd.
If $AC$ is even, then the proof is identical, using  $F(x_1, x_2) =x_1 x_2$
and using part (b) of lemma \ref{l-congruences} instead of part (a).
\end{proof}
\begin{lemma}
(a) Let $A, B, C$ be odd integers. Let $r \geq 1$. Then there exists a matrix $S \in M_2(\ZZ)$ such that
$S \bigl( \begin{smallmatrix} 2 & 1 \\ 1 & 2 \end{smallmatrix} \bigr) S^{\op{tr}} \equiv 
\bigl( \begin{smallmatrix} 2A & B \\ B & 2C \end{smallmatrix} \bigr) \bmod 2^r$
and $S \equiv I \bmod 2$.
 \par
 (b) Let $A, B, C$ be integers such that $A C$ is even and $B$ is odd.
 Let $r \geq 1$. Then there exists a matrix $S \in M_2(\ZZ)$ such that
$S \bigl( \begin{smallmatrix} 0 & 1 \\ 1 & 0 \end{smallmatrix} \bigr) S^{\op{tr}} \equiv 
\bigl( \begin{smallmatrix} 2A & B \\ B & 2C \end{smallmatrix} \bigr) \bmod 2^r$
and $S \equiv  \bigl( \begin{smallmatrix} A & 1 \\ 1 & C \end{smallmatrix} \bigr) \bmod 2$. 
\label{l-rank-two-irr-b}
\end{lemma}
\begin{proof}
(a) The congruences in part (a) translate into $A(s) \equiv A \bmod 2^{r-1}, B(s) \equiv B \bmod 2^r, C(s) = C \bmod 2^{r-1}$
where $A(s), B(s), C(s)$ are as in lemma \ref{l-congruences} (a). Part (a) follows from lemma \ref{l-congruences}.
Similarly part (b) follows from part (b) of lemma \ref{l-congruences}.  
\end{proof}
\begin{proof}[Proof of theorem \ref{t-wall}]
(a) Let $(G, b)$ be a non-degenerate discriminant form. 
It suffices to decompose $(G,b)$ into irreducibles when $G$ is a $p$-group for some prime $p$. First suppose
$p$ is odd.  
From lemma \ref{l-diagonalization}, it follows that there exists $f_1, \dotsb, f_n \in G$ such that $G = \oplus \langle f_j \rangle$ and 
$\op{Gram}_b( f_1, \dotsb, f_n) = \op{diag}( p^{-r_1} \epsilon_1, \dotsb, p^{-r_n} \epsilon_n)$
with  $r_1\geq r_2 \geq \dotsb \geq r_n \geq 0$, 
$\epsilon_j \in \lbrace 1, u_p \rbrace$. Since $(G,b)$ it non-degenerate, it follows that we must have $\op{order}(f_j) = p^{r_j}$ for all $j$.
Thus $(G, b)$ is orthogonal direct sum of the rank one discriminant forms $(\langle f_j \rangle, b \vert_{\langle f_j \rangle} )$
and each of these are of type $A$ or $B$. This completes the argument for odd $p$.
\par
Now we consider the case $p = 2$. From lemma \ref{l-diagonalization2}, it follows that there exists
$f_1, \dotsb, f_n \in G$ such that $G = \oplus \langle f_j \rangle$ and 
$\op{Gram}_b( f_1, \dotsb, f_n)$ is block diagonal with blocks of size one or two as given in lemma
\ref{l-diagonalization2}. Accordingly $(G,b)$ is an orthogonal direct sum of rank one or two discriminant
forms spanned by one or two of the $f_j$'s. The rank one forms among these are clearly of type $A$, $B$, $C$
or $D$. The Gram matrix of a rank two piece has the form
$2^{-r} \bigl( \begin{smallmatrix} 2 a & b \\ b & 2 c \end{smallmatrix} \bigr)$. 
Lemma \ref{l-rank-two-irr-b} shows that such a rank two piece is either of type $E$ or $F$.
%
\par
(b) Let $(G, q)$ be a metric group. By part (a), $(G, \partial q)$ is an orthogonal direct sum of irreducible forms $(G_j, b_j)$.
Each $G_j$ is a homogeneous $p$-group of rank $1$ or $2$.
Further $G_j$ can have rank two only if $p = 2$. It follows that $(G, q)$ is also an orthogonal
direct sum of $(G_j, q_j)$ where $q_j  = q \vert_{G_j}$. The rank one forms are clearly of type $A$, $B$, $C$ or $D$.
The rank two forms either decompose into two rank one forms or they are irreducible as metric groups. In the later
case, lemma \ref{l-rank-two-irr-q} shows that $(G_j, q_j)$ is of type $E$ or $F$.
\end{proof}
%
%
%
%
%
%
%

\end{document}